\documentclass[12pt,leqno,oneside]{amsart}
\usepackage{mathrsfs,dsfont}
\usepackage{amsmath,amstext,amsthm,amssymb}
\usepackage{charter}
\usepackage{typearea}
\usepackage{pdfsync}

\usepackage[width=6.4in,height=8.5in]
{geometry}

\newtheorem{Theorem}{Theorem}[section]
\newtheorem{Proposition}[Theorem]{Proposition}
\newtheorem{Lemma}[Theorem]{Lemma}
\newtheorem{Corollary}[Theorem]{Corollary}
\theoremstyle{definition}
\newtheorem{Definition}[Theorem]{Definition}

\newtheorem{Remark}[Theorem]{Remark}

\def\re{{\mathbf {Re\,}}}

\newcommand{\db}{\overline\partial}

\newcommand{\wi}{\widetilde}
\DeclareMathOperator{\ric}{Ric}
\DeclareMathOperator{\codim}{codim}

\DeclareMathOperator{\supp}{supp}
\DeclareMathOperator{\inte}{int}
\newcommand{\cali}[1]{\mathscr{#1}}
\newcommand{\cO}{\cali{O}} 
\newcommand{\cM}{\cali{M}}\newcommand{\cT}{\cali{T}}
\newcommand{\cC}{\cali{C}}\newcommand{\cA}{\cali{A}}
\newcommand{\field}[1]{\mathbb{#1}}
\newcommand{\Z}{\field{Z}}
\newcommand{\R}{\field{R}}
\newcommand{\C}{\field{C}}
\newcommand{\N}{\field{N}}

\newcommand{\comment}[1]{}

\begin{document}

\title[Equidistribution for sequences of line bundles  
on normal K\"ahler spaces]{Equidistribution for sequences
of line bundles  on normal K\"ahler spaces}
\author{Dan Coman}
\thanks{D.\ Coman is partially supported by the 
NSF Grant DMS-1300157}
\address{Department of Mathematics, 
Syracuse University, Syracuse, NY 13244-1150, USA}
\email{dcoman@syr.edu}
\author{Xiaonan Ma}
\address{Institut Universitaire de France \& 
Universit\'e Paris Diderot - Paris 7, UFR de Math\'ematiques, 
Case 7012,
75205 Paris Cedex 13, France}
\email{xiaonan.ma@imj-prg.fr}
\thanks{X.\ Ma is partially supported by Institut Universitaire 
de France}
\author{George Marinescu}
\address{Universit{\"a}t zu K{\"o}ln, Mathematisches Institut, 
Weyertal 86-90, 50931 K{\"o}ln, Deutschland   \& 
Institute of Mathematics `Simion Stoilow', Romanian Academy, 
Bucharest, Romania}
\email{gmarines@math.uni-koeln.de}
\thanks{G.\ Marinescu is partially supported by DFG funded projects 
SFB/TR 12, MA 2469/2-2}
\thanks{Funded through the Institutional Strategy of
the University of Cologne within the German Excellence Initiative}
\date{January 1, 2016}

\pagestyle{myheadings}

\begin{abstract} 
We study the asymptotics of Fubini-Study currents and 
zeros of random holomorphic sections associated to 
a sequence of singular Hermitian line bundles on a
compact normal K\"ahler complex space.
\end{abstract}

\maketitle
\tableofcontents

\section{Introduction}\label{S:intro}

\par In this paper we continue the study of equidistribution 
of Fubini-Study currents and zeros of sequences of holomorphic
sections of singular Hermitian holomorphic bundles started in 
\cite{CM11,CM13,CM13b}. We generalize our previous results 
in two directions. On the one hand, we allow the base space to 
be singular and work over K\"ahler spaces. On the other hand, 
we consider sequences $(L_p,h_p)$, $p\geq1$, 
of singular Hermitian holomorphic line bundles whose 
Chern curvature satisfy a natural growth condition, 
instead of sequences of powers $(L^p,h^p)$ of a fixed
line bundle $(L,h)$. 

\par Recall that by the results of \cite{Ti90} 
(see also \cite[Section 5.3]{MM07}), if $(X,\omega)$ is 
a compact K\"ahler manifold whose K\"ahler form is integral 
and $(L,h)$ is a prequantum line bundle (i.e.,\ the Chern curvature 
form $c_1(L,h)$ equals $\omega$), then the normalized
Fubini-Study currents $\frac1p\gamma_p$ associated to 
$H^0(X,L^p)$ converge in the $\cC^2$ topology to $\omega$. 
This result can be applied to describe the asymptotic distribution 
of zeros of random sequences of holomorphic sections. 
Indeed, it is shown in \cite{ShZ99} 
(see also \cite{NoVo:98,DS06,ShZ08,Sh08,DMS}) 
that for almost all sequences 
$\{\sigma_p\in H^0(X,L^p)\}_{p\geq1}$ the normalized 
zero-currents $\frac1p[\sigma_p=0]$ converge weakly to 
$\omega$ on $X$.   
This means that the K\"ahler form $\omega$ can be approximated 
by various algebraic/analytic objects in the semiclassical limit 
$p\to\infty$\,. Some important technical tools for higher 
dimension used in the these works were introduced in \cite{FS95}. 
We note that statistics of zeros of sections and hypersurfaces
have been studied also in the context of real manifolds and
real vector bundles, see e.g.\ \cite{GW14,NS14}.

\par In \cite{CM11,CM13,CM13b} we relaxed the smoothness
condition on $\omega$ and assumed that $\omega$ is merely 
an integral K\"ahler current. Then there exists a holomorphic
line bundle $(L,h)$ endowed with a singular Hermitian metric 
such that $c_1(L,h)$ equals $\omega$ as currents. 
We established the above approximation results in the sense 
of currents by working with the spaces of square integrable 
holomorphic sections $H^0_{(2)}(X,L^p)$. 
The setting in \cite{CM11,CM13,CM13b} allows to deal with 
several singular K\"ahler metrics such as singular metrics on 
big line bundles, metrics with Poincar\'e growth, K\"ahler-Einstein 
metrics singular along a divisor (e.\,g.\ good metrics
in the sense of Mumford on toroidal compactifications of 
arithmetic quotiens \cite{AMRT:10}) or metrics on 
orbifold line bundles.

\par We consider in this paper the following setting:

\smallskip

(A) $(X,\omega)$ is a compact (reduced) normal K\"ahler space 
of pure dimension $n$, $X_{\rm reg}$ denotes the set of 
regular points of $X$, and $X_{\rm sing}$ denotes the set 
of singular points of $X$.

\smallskip

(B) $(L_p,h_p)$, $p\geq1$, is a sequence of holomorphic line 
bundles on $X$ with singular Hermitian metrics $h_p$ whose 
curvature currents verify 
\begin{equation}\label{e:pc}
c_1(L_p,h_p)\geq a_p\,\omega \, \text{ on $X$, where 
$a_p>0$ and } \lim_{p\to\infty}a_p=\infty.
\end{equation}
Let $A_p=\int_Xc_1(L_p,h_p)\wedge\omega^{n-1}$. 
If $X_{\rm sing}\neq\emptyset$ we assume moreover that 
\begin{equation}\label{e:domin0}
\exists\,T_0\in\cT(X) \text{ such that } 
c_1(L_p,h_p)\leq A_pT_0\,,\;\forall\,p\geq1\,.
\end{equation}

\smallskip
\par Here we consider currents on the analytic space $X$
in the sense of \cite{D85}, and $\cT(X)$ is the set of positive 
closed currents of bidegree $(1,1)$ on $X$ which have a local 
plurisubharmonic (psh) potential in the neighborhood of each point
of $X$ (see Section \ref{S:prelim}). 

\par We let $H^0_{(2)}(X,L_p)$ be the Bergman space 
of $L^2$-holomorphic sections of $L_p$ relative to the metric 
$h_p$ and the volume form induced by $\omega$ on $X$, 
\begin{equation}\label{e:bs}
H^0_{(2)}(X,L_p)=\left\{S\in H^0(X,L_p):\,
\|S\|_p^2:=\int_{X_{\rm reg}}|S|^2_{h_p}\,\frac{\omega^n}{n!}
<\infty\right\}\,,
\end{equation}
endowed with the obvious inner product. Furthermore, 
we denote by $P_p$, resp. $\gamma_p$, the 
Bergman kernel function, resp. the Fubini-Study current, 
of the space $H^0_{(2)}(X,L_p)$, which are defined as follows. 
For $p\geq1$, let $S_1^p,\dots,S_{d_p}^p$ be an orthonormal 
basis of $H^0_{(2)}(X,L_p)$. If $x\in X$ and $e_p$ is a
local holomorphic frame of $L_p$ in a neighborhood $U_p$ of $x$ 
we write $S_j^p=s_j^pe_p$, where $s_j^p\in\mathcal O_X(U_p)$. 
Then 
\begin{equation}\label{e:BFS1}
P_p(x)=\sum_{j=1}^{d_p}|S^p_j(x)|_{h_p}^2\;,\;\;
\gamma_p\vert_{U_p}=\frac{1}{2}\,dd^c
\log\left(\sum_{j=1}^{d_p}|s_j^p|^2\right)\,,
\end{equation}
where $d^c=\frac{1}{2\pi i}(\partial-\overline\partial)$. 
Note that $P_p,\,\gamma_p$ are independent of the choice 
of basis. 

\par Our main results are the following theorems:

\begin{Theorem}\label{T:mt1} If $(X,\omega)$, $(L_p,h_p)$, 
    $p\geq1$, verify assumptions (A)-(B), then:

\smallskip

\par (i) $\frac{1}{A_p}\,\log P_p\to 0$ as $p\to\infty$,
in $L^1(X,\omega^n)$. 

\smallskip

\par (ii) $\frac{1}{A_p}\,(\gamma_p-c_1(L_p,h_p))\to 0$ as 
$p\to\infty$, in the weak sense of currents on $X$.

\smallskip

Moreover, if $\frac{1}{A_p}\,c_1(L_p,h_p)\to T$
for some positive closed current $T$ of bidegree $(1,1)$ on $X$, 
then $\frac{1}{A_p}\,\gamma_p\to T$ as $p\to\infty$, 
in the weak sense of currents on $X$.
\end{Theorem}

\par When $X$ is smooth we obtain:

\begin{Theorem}\label{T:mt2} Let $(X,\omega)$ be a 
    compact K\"ahler manifold of dimension $n$ and $(L_p,h_p)$,
    $p\geq1$, be a sequence of singular Hermitian holomorphic 
    line bundles on $X$ which satisfy  
    $c_1(L_p,h_p)\geq a_p\,\omega$, where $a_p>0$ and 
    $\lim_{p\to\infty}a_p=\infty$\,.
    If $A_p=\int_Xc_1(L_p,h_p)\wedge\omega^{n-1}$ then 
    $\frac{1}{A_p}\,\log P_p\to 0$ in $L^1(X,\omega^n)$ 
    and $\frac{1}{A_p}\,(\gamma_p-c_1(L_p,h_p))\to 0$ weakly 
    on $X$, as $p\to\infty$\,.
\end{Theorem}

\par If $(L_p,h_p)=(L^p,h^p)$, where $(L,h)$ is a fixed singular 
Hermitian holomorphic line bundle, we obtain analogues of the 
equidistributions results from \cite{ShZ99,DS06,CM11,CM13,CM13b}
for compact normal K\"ahler spaces. Note that in this case 
hypothesis \eqref{e:domin0} is automatically verified as 
$c_1(L^p,h^p)=p\,c_1(L,h)$, so we can take $A_p=p$ and 
$T_0=c_1(L,h)$.  

\par In the case $(L_p,h_p)=(L^p,h^p)$, Theorem \ref{T:mt1} 
gives semiclassical approximation results for integral K\"ahler 
currents. In order to deal with the non-integral case, we consider
those currents $T$ which can be approximated by the curvatures 
of a sequence $(L_p,h_p)$, $p\geq1$, of singular Hermitian 
holomorphic line bundles. Such a sequence can be thought as a 
``prequantization'' of the non-integral positive closed 
$(1,1)$ current $T$. Theorems \ref{T:mt1}-\ref{T:mt2} 
and their consequences are a manifestation of the quantum 
ergodicity in this context. For details, see 
Sections \ref{S:zeros}, \ref{S:appl}, and Theorem \ref{T:approx}. 
Another interesting situation is when $L_p$ equals a product of
tensor powers of several holomorphic line bundles, 
$L_p=F_1^{m_{1,p}}\otimes\ldots\otimes F_k^{m_{k,p}}$, 
where $\{m_{j,p}\}_p$\,, $1\leq j\leq k$, are sequences in $\N$ 
such that $m_{j,p}=r_j\,p+o(p)$ as $p\to\infty$, where $r_j>0$ 
are given. This means that $(m_{1,p},\ldots,m_{k,p})\in\N^k$ 
approximate the semiclassical ray 
$\R_{>0}\cdot(r_1,\ldots,r_k)\in\R_{>0}^k$ with a remainder 
$o(p)$, as $p\to\infty$. For details see Corollary \ref{C:prod}.


\smallskip

\par We consider further the situation when the metrics $h_p$ on 
the bundles $L_p$ are smooth. In the case of powers $(L_p,h_p)=(L^p,h^p)$ of an ample line bundle $(L,h)$, the first order asymptotics of the Bergman kernel function was showed by Tian \cite{Ti90}. A simple proof for this was given by Berndtsson in \cite[Section 2]{Be03}. Adapting his methods to our situation we prove the following:
\begin{Theorem}\label{T:mt3}
Let $(X,\omega)$ be a compact K\"ahler manifold of dimension $n$. 
Let $(L_p,h_p)$, $p\geq1$, be a sequence of holomorphic 
line bundles on $X$ with Hermitian metrics $h_p$ of class $\cC^3$ 
whose curvature forms verify \eqref{e:pc} and such that 
\begin{equation}\label{e:3d}
\varepsilon_p:=\|h_p\|_3^{1/3}a_p^{-1/2}\to0\;
\text{ as }p\to\infty\,.
\end{equation}
Then there exist $C>0$ depending only on $(X,\omega)$ 
and $p_0\in\mathbb N$ such that  
\begin{equation}\label{e:Bexp00}
\left|P_p(x)\,\frac{\omega^n_x}{c_1(L_p,h_p)^n_x}-1\right|\leq C\varepsilon_p^{2/3}
\end{equation}
holds for every $x\in X$ and $p>p_0$. 
\end{Theorem}
\par Here $\|h_p\|_3$ denotes the sup-norm of the derivatives of $h_p$ of order at most three with respect to a reference cover of $X$ as defined in Section \ref{SS:refcov}. 
Theorem \ref{T:mt3} is a generalization of the first order 
asymptotic expansion of the Bergman kernel
\cite{Be03,Ca99,DLM04a,MM07,MM08,Ru98,Ti90,Z98}
for $(L_p,h_p)=(L^p,h^p)$, 
where $(L,h)$ is a positive line bundle with smooth metric $h$
(see Remark \ref{R:Bka}).
\par The paper is organized as follows. In Section \ref{S:prelim} 
we collect the necessary technical facts about complex spaces, 
plurisubharmonic functions, currents and singular Hermitian metrics 
on holomorphic line bundles. Section \ref{S:pfmt} is devoted to the 
proofs of Theorems \ref{T:mt1}, \ref{T:mt2}, \ref{T:mt3}.
In Section \ref{S:zeros} we give some applications of these
theorems to the equidistribution of zeros of random sequences 
of holomorphic sections and to the approximation of certain positive 
closed currents on $X$ by currents of integration along zeros of 
holomorphic sections. In Section \ref{S:appl} we specialize these 
results to the cases when $(L_p,h_p)$ are the powers of a single 
line bundle, or tensor products of powers of several line bundles. 
We also show that some interesting cases like Satake-Baily-Borel 
compactifications, 
singular K\"ahler-Einstein metrics on varieties of general type 
or K\"ahler metrics with conical singularities 
fit into this framework.

\section{Preliminaries}\label{S:prelim} 

\par We recall here a few notions and results that will be 
needed throughout the paper. 

\subsection{Plurisubharmonic functions and currents on analytic 
spaces}\label{SS:psh} 
Let $X$ be a complex space. A chart $(U,\tau,V)$ on $X$ is a 
triple consisting of an open set $U\subset X$, a closed complex 
space $V\subset G\subset\C^N$ in an open set $G$ of $\C^N$ 
and a biholomorphic map $\tau:U\to V$ (in the category of complex 
spaces). The map $\tau:U\to G\subset\C^N$ is called a local 
embedding of the complex space $X$. We write
\[X=X_{\rm reg}\cup X_{\rm sing}\,,\]
where $X_{\rm reg}$ (resp.\ $X_{\rm sing}$) is the set of 
regular (resp.\ singular) points of $X$. 
Recall that a reduced complex space $(X,\cO)$ is called normal
if for every $x\in X$ the local ring $\cO_x$ is integrally closed in 
its quotient field $\cM_x$. Every normal complex space is 
locally irreducible and locally pure dimensional, 
cf.\ \cite[p.\,125]{GR84}, $X_{\rm sing}$ is a closed complex 
subspace of $X$ with $\codim X_{\rm sing}\geq2$. Moreover, 
Riemann's second extension theorem holds on normal complex 
spaces \cite[p.\,143]{GR84}. In particular, every holomorphic 
function on $X_{\rm reg}$ extends uniquely to a holomorphic 
function on $X$.   

Let $X$ be a complex space. A continuous (resp.\ smooth) function 
on $X$ is a function $\varphi:X\to\C$ such that for every 
$x\in X$ there exists a local embedding $\tau:U\to G\subset\C^N$ 
with $x\in U$ and a continuous (resp.\ smooth) function 
$\widetilde\varphi:G\to\C$ such that 
$\varphi|_U=\widetilde\varphi\circ\tau$. 
  
A \emph{(strictly) plurisubharmonic (psh)} function on $X$ 
is a function $\varphi:X\to[-\infty,\infty)$ such that for every 
$x\in X$ there exists a local embedding $\tau:U\to G\subset\C^N$ 
with $x\in U$ and a (strictly) psh function 
$\widetilde\varphi:G\to[-\infty,\infty)$ such that 
$\varphi|_U=\widetilde\varphi\circ\tau$. If $\widetilde\varphi$ 
can be chosen continuous (resp.\ smooth), then $\varphi$ is 
called a continuous (resp.\ smooth) psh function. 
The definition is independent of the chart, as is seen from 
\cite[Lemma\,4]{Nar62}. 
It is clear that a continuous psh function is continuous; 
by a theorem of Richberg \cite{Ric68} the converse also holds true,
i.e.,\ a continuous function which is (strictly) psh is also continuous 
(strictly) psh. The analogue of Riemann's second extension theorem 
for psh functions holds on normal complex spaces 
\cite[Satz\,4]{GR56}. In particular, every psh function on 
$X_{\rm reg}$ extends uniquely to a psh function on $X$. 
We let $PSH(X)$ denote the set of psh functions on $X$, 
and refer to \cite{GR56}, \cite{Nar62}, \cite{FN80}, \cite{D85} 
for the properties of psh functions on $X$. We recall here that 
psh functions on $X$ are locally integrable with respect to the 
area measure on $X$ given by any local embedding 
$\tau:U\to G\subset\C^N$ \cite[Proposition 1.8]{D85}.

Let $X$ be a complex space of pure dimension $n$. 
We consider currents on $X$ as defined in \cite{D85}.
The sheaf of smooth $(p,q)$-forms on $X$ is defined at first locally. 
Let $\tau:U\to G\subset\C^N$ be a local embedding. We define 
$\Omega^{p,q}(U)$ to be the image of the morphism
$\tau^*:\Omega^{p,q}(G)\to\Omega^{p,q}(U_{\rm reg})$. 
It can be easily seen that there exists a sheaf $\Omega^{p,q}$ 
on $X$ whose space of sections on any domain $U$ of local 
embedding is $\Omega^{p,q}(U)$. Let 
$\mathcal D^{p,q}(X)\subset\Omega^{p,q}(X)$ be the space
of forms with compact support, endowed with the inductive limit
topology. The dual $\mathcal D'_{p,q}(X)$ of 
$\mathcal D_{p,q}(X)$ is the space of currents of bidimension 
$(p,q)$, or bidegree $(n-p,n-q)$, on $X$.     
In particular, if $v\in PSH(X)$ then 
$dd^cv\in\mathcal D'_{n-1,n-1}(X)$ is positive and closed. 

\par Let $\cT(X)$ be the space of positive closed currents 
of bidegree $(1,1)$ on $X$ which have local psh potentials: 
$T\in\cT(X)$ if every $x\in X$ has a neighborhood $U$
(depending on $T$) such that there exists a psh function $v$ on 
$U$ with $T=dd^cv$ on $U\cap X_{\rm reg}$. 
Most of the currents considered here, like the curvature currents 
$c_1(L_p,h_p)$ and the Fubini-Study currents $\gamma_p$ from 
Theorem \ref{T:mt1}, belong to $\cT(X)$. Suppose now that 
$Y$ is a normal analytic space, $f:Y\longrightarrow X$ is 
a holomorphic map, and $T\in\cT(X)$ is such that if $v$ is a local 
psh potential of $T$ then $v\circ f$ is not identically $-\infty$ 
on an open set of $Y$. Then the pull-back $f^* T\in\cT(Y)$ is 
a well-defined current whose local psh potentials are $v\circ f$. 
Some interesting open questions that we will not pursue here are 
the following: Does every positive closed current of bidegree 
$(1,1)$ on $X$ belong to $\cT(X)$? Is $\cT(X)$ closed in 
the weak$^*$ topology on currents? If $T_k,\,T\in\cT(X)$ 
and $T_k\to T$ weakly on $X$, does $\{f^* T_k\}$ converge to 
$f^* T$ weakly on $Y$?

\par A K\"ahler form on $X$ is a current $\omega\in\cT(X)$ 
whose local potentials 
extend to smooth strictly psh functions in local embeddings of $X$ 
to Euclidean spaces. We call $X$ a K\"ahler space if $X$ admits
 a K\"ahler form (see also \cite[p.\,346]{Gra:62}, 
 \cite{O87}, \cite[Sec. 5]{EGZ09}). 
 A K\"ahler form is a particular case of a Hermitian form on 
 a complex space. 
 Recall that a Hermitian form on a complex manifold is a smooth 
 positive $(1,1)$-form and can be identified to a Hermitian metric. 
Now, a Hermitian form on a complex space $X$ is defined as 
a smooth $(1,1)$-form $\omega$ on $X$ such that for every point 
$x\in X$ there exists a local embedding 
$\tau:U\to G\subset\C^N$ 
with $x\in U$ and a Hermitian form $\widetilde\omega$ on $G$ with 
  $\omega=\tau^*\widetilde\omega$ on $U\cap X_{\rm reg}$\,.
A Hermitian form on a paracompact complex space $X$ is 
constructed as usual by a partition of unity argument. 
A Hermitian form $\omega$ on $X$ clearly induces a Hermitian form 
in the usual sense (and thus a Hermitian metric) on $X_{\rm reg}$\,. 
Note that $\omega^n/n!$ gives locally an area measure on $X$.

\subsection{Singular Hermitian holomorphic line bundles 
on analytic spaces}\label{SS:lb} 
Let $L$ be a holomorphic line bundle on a normal K\"ahler space
$(X,\omega)$. 
The notion of singular Hermitian metric $h$ on $L$ is defined 
exactly as in the smooth case 
(see \cite{D90}, \cite[p.\,97]{MM07})): if $e_\alpha$
  is a holomorphic frame of $L$ over an open set 
  $U_\alpha\subset X$
   then $|e_\alpha|^2_h=e^{-2\varphi_\alpha}$ where 
   $\varphi_\alpha\in L^1_{loc}(U_\alpha,\omega^n)$.
 If $g_{\alpha\beta}=e_\beta/e_\alpha\in\mathcal O^*_X
 (U_\alpha\cap U_\beta)$ 
    are the transition functions of $L$ then 
    $\varphi_\alpha=\varphi_\beta+\log|g_{\alpha\beta}|$. 
The curvature current $c_1(L,h)\in\mathcal D'_{n-1,n-1}(X)$ of $h$
 is defined by $c_1(L,h)=dd^c\varphi_\alpha$ on 
 $U_\alpha\cap X_{\rm reg}$. We will denote by $h^p$ 
 the singular Hermitian metric induced by $h$ on 
$L^p:=L^{\otimes p}$. If $c_1(L,h)\geq0$ then the weight 
$\varphi_\alpha$
is psh on $U_\alpha\cap X_{\rm reg}$ and since $X$ is normal
it extends to a psh function on $U_\alpha$ \cite[Satz\,4]{GR56}, 
        hence $c_1(L,h)\in\cT(X)$. 
\par Let $L$ be a holomorphic line bundle on a compact normal 
K\"ahler space $(X,\omega)$. Then the space $H^0(X,L)$ of 
holomorphic sections of $L$ is finite dimensional. 
This is a special case of the Cartan-Serre
 finiteness theorem; an elementary proof using the 
 Schwarz lemma can be found in 
 \cite[Th\'eor\`eme 1,\,p.27]{An:63}. The space $H^0_{(2)}(X,L)$
  defined as in \eqref{e:bs} is therefore also finite dimensional.

\par If $P_p$ and $\gamma_p$ are the Bergman kernel functions, 
resp. the Fubini-Study currents, of the spaces $H^0_{(2)}(X,L_p)$ 
from Theorem \ref{T:mt1}, it follows from \eqref{e:BFS1} that 
$\log P_p\in L^1(X,\omega^n)$ and 
\begin{equation}\label{e:BFS2}
\gamma_p-c_1(L_p,h_p)=\frac{1}{2}\,dd^c\log P_p\,.
\end{equation}
Moreover, as in \cite{CM11,CM13}, one has the following variational formula,  
$$P_p(x)=\max\big\{|S(x)|^2_{h_p}:\,S\in H^0_{(2)}(X,L_p),\;
\|S\|_p=1\big\}.$$
This is valid for all $x\in X$ such that $\varphi_p(x)>-\infty$, where $\varphi_p$ is a local weight of the metric $h_p$ near $x$.

\subsection{Resolution of singularities}\label{SS:desing} 
Bierstone and Milman constructed a resolution of singularities of a 
compact analytic space $X$, $\pi:\widetilde X\longrightarrow X$, 
by a finite sequence of blow-ups with smooth center 
$\sigma_j:X_{j+1}\longrightarrow X_j$, $X_0=X$, with 
the property that for any local embedding 
$X\vert_U\hookrightarrow{\mathbb C}^N$ this sequence of 
blow-ups is induced by the embedded desingularization of 
$X\vert_U$ \cite[Theorem 13.2]{BM97}. In \cite[sec. 6]{GM06}
it is shown that the embedded desingularization of 
$X\vert_U\hookrightarrow{\mathbb C}^N$ by a finite sequence 
of blow-ups with smooth center is equivalent to a single blow-up 
along a coherent sheaf of ideals $\mathcal I$ whose support is 
$X_{\rm sing}\vert_U$. It follows that every point $x\in X$ has 
a neighborhood $U\subset X$ for which there exists an ideal 
$\mathcal I_U$ generated by finitely many holomorphic functions 
on $U$ such that $\pi:\pi^{-1}(U)\longrightarrow U$ is equivalent 
to the blow-up of $X\vert_U$ along $\mathcal I_U$. 

\par We fix throughout the paper a resolution of singularities 
$\pi:\widetilde X\longrightarrow X$ of our compact normal space 
$X$ as described above, and we denote by 
$\Sigma=\pi^{-1}(X_{\rm sing})$ the exceptional divisor. 
Note that $\pi:\widetilde X\setminus\Sigma\longrightarrow 
X_{\rm reg}$ is a biholomorphism. We will need to consider the 
singular Hermitian holomorphic line bundles $(\pi^* L_p,\pi^* h_p)$ 
obtained by pulling back $(L_p,h_p)$ to $\widetilde X$ by the map 
$\pi$, and their spaces of $L^2$-holomorphic sections 
$$H^0_{(2)}(\widetilde X,\pi^* L_p)=
\left\{\widetilde S\in H^0(\widetilde X,\pi^* L_p):\,
\int_{\widetilde X}|\widetilde S|^2_{\pi^* h_p}\,\frac{\pi^*\omega^n}{n!}
<\infty\right\}.$$

\begin{Lemma}\label{L:bPp} The map 
    $\pi^*:H^0_{(2)}(X,L_p)\longrightarrow 
    H^0_{(2)}(\widetilde X,\pi^* L_p)$ is an isometry and 
    the Bergman kernel function of 
    $H^0_{(2)}(\widetilde X,\pi^* L_p)$ is 
    $\widetilde P_p=P_p\circ\pi$. 
\end{Lemma}

\begin{proof} Let $S^p_1,\dots,S^p_{d_p}$ be an orthonormal 
    basis of $H^0_{(2)}(X,L_p)$ and $\widetilde S^p_j=\pi^* S^p_j$ 
    be the induced sections of $\pi^* L^p$. Then 
    $|\widetilde S^p_j|_{\pi^* h_p}=|S^p_j|_{h_p}\circ\pi$ and 
$$\int_{\widetilde X}|\widetilde S^p_j|^2_{\pi^* h_p}\,
\frac{\pi^*\omega^n}{n!}=\int_{\widetilde X\setminus\Sigma}
|\widetilde S^p_j|^2_{\pi^* h_p}\,\frac{\pi^*\omega^n}{n!}
=\int_{X_{\rm reg}}|S^p_j|^2_{h_p}\,\frac{\omega^n}{n!}=1.$$
Suppose now that 
$\widetilde S\in H^0_{(2)}(\widetilde X,\pi^* L_p)$ is orthogonal 
to all $\widetilde S^p_j$ and let $S$ be the induced section on 
$X_{\rm reg}$. Then, since $X$ is normal, Riemann's second 
extension theorem \cite[p.\,143]{GR84} shows that $S$ extends 
to a section $S\in H^0(X,L_p)$. By the preceding argument, 
$S\in H^0_{(2)}(X,L_p)$ is orthogonal to all $S^p_j$, hence $S=0$
and $\widetilde S=0$. It follows that $\{\widetilde S^p_j\}$ is an 
orthonormal basis of $H^0_{(2)}(\widetilde X,\pi^* L_p)$ and 
$$\widetilde P_p=\sum_{j=1}^{d_p}|\widetilde S^p_j|^2_{\pi^* h_p}
=\sum_{j=1}^{d_p}|S^p_j|^2_{h_p}\circ\pi=P_p\circ\pi.$$
\end{proof}

\par Adapting the proof of \cite[Lemma 1]{M67} to our situation 
we obtain the following lemma, whose proof is included for
the convenience of the reader.

\begin{Lemma}\label{L:Moishezon} 
Let $(X,\omega)$ be a compact (reduced) Hermitian space and 
$\pi:\widetilde X\longrightarrow X$ be a resolution of singularities 
as described above, with exceptional divisor $\Sigma$\,. 
Then there exists a smooth Hermitian metric $\theta$ on 
$F=\mathcal O_{\widetilde X}(-\Sigma)$ and $C>0$ such that 
$\Omega=C\pi^*\omega+c_1(F,\theta)$ is a Hermitian form on 
$\widetilde X$ and $\Omega\geq\pi^*\omega$. If $\omega$ 
is K\"ahler, then $\Omega$ is K\"ahler, too.
\end{Lemma}

\begin{proof} As in the proof of \cite[Lemma 1]{M67}, we can 
    find an open cover $X=\bigcup_{k=1}^NU_k$ with the following 
    properties:

\medskip

\par (i) There exist a local embedding $\tau_k:U_k\to B_k$ into 
a ball $B_k\subset{\mathbb C}^{l_k}$ and a smooth Hermitian form 
$\omega_k$ on $B_k$ such that $\omega=\tau_k^*\omega_k$
on $U_k\cap X_{\rm reg}$. Hence we may assume that 
$U_k\subset B_k$ and $\omega$ is the restriction of $\omega_k$
to $U_k\cap X_{\rm reg}$. Choose a strictly psh function $\eta_k$
on $B_k$ such that $\omega_k\geq dd^c\eta_k$ on $B_k$.
If $\omega$ is K\"ahler, we choose $\eta_k$ such that 
$\omega_k=dd^c\eta_k$ on $B_k$, hence $\omega=dd^c\eta_k$ 
on $U_k\cap X_{\rm reg}$\,. 

\par (ii) There exist finitely many holomorphic functions 
$f_{0k},\dots,f_{N_kk}\in\mathcal O_X(U_k)$ such that over 
$U_k$ the map $\pi:\pi^{-1}(U_k)\longrightarrow U_k$ and the line
bundle $\mathcal O_{\widetilde X}(\Sigma)\vert_{\pi^{-1}(U_k)}$
are described as follows.
If $\Gamma_k\subset U_k\times{\mathbb P}^{N_k}$ is the graph 
of the meromorphic map $x\mapsto[f_{0k}(x):\ldots:f_{N_kk}(x)]$ 
with canonical projections $\pi_{1k}:\Gamma_k\longrightarrow U_k$ 
and $\pi_{2k}:\Gamma_k\longrightarrow{\mathbb P}^{N_k}$ then 
$\pi:\pi^{-1}(U_k)\longrightarrow U_k$ coincides with 
$\pi_{1k}:\Gamma_k\longrightarrow U_k$ and 
$\mathcal O_{\widetilde X}(\Sigma)\vert_{\pi^{-1}(U_k)}$ 
can be identified with the line bundle 
$\pi^*_{2k}\mathcal O_{{\mathbb P}^{N_k}}(-1)$ on $\Gamma_k$.

\medskip

Let $[y_{0k}:\ldots:y_{N_kk}]$ denote the homogeneous 
coordinates on ${\mathbb P}^{N_k}$, 
$G_{ik}=\{y_{ik}\neq0\}\subset{\mathbb P}^{N_k}$, 
$\widetilde G_{ik}=\pi^{-1}_{2k}(G_{ik})\subset\Gamma_k
=\pi^{-1}(U_k)$. The transition functions of
$\mathcal O_{{\mathbb P}^{N_k}}(-1)$ are $y_{ik}/y_{jk}$
and we endow $\mathcal O_{{\mathbb P}^{N_k}}(-1)$ with a 
smooth Hermitian metric with weights $\psi'_{ik}$ on $G_{ik}$ such
that $-\psi'_{ik}$ is strictly psh. The corresponding transition 
functions of $\mathcal O_{\widetilde X}(\Sigma)
\vert_{\pi^{-1}(U_k)}$ are
$\widetilde g_{ij,k}=(f_{ik}\circ\pi_{1k})/(f_{jk}\circ\pi_{1k})$ 
on $\widetilde G_{ik}\cap\widetilde G_{jk}$, and we denote by 
$\psi_{ik}$ the weights of the induced Hermitian metric. 
Then $-\psi_{ik}$ is a smooth psh function on $\widetilde G_{ik}$ 
and $\psi_{ik}=\psi_{jk}+\log|\widetilde g_{ij,k}|$ on 
$\widetilde G_{ik}\cap\widetilde G_{jk}$.

\par Consider the open cover 
$\{\widetilde G_{ik}:\,1\leq k\leq N,\,0\leq i\leq N_k\}$ of 
$\widetilde X$. Note that 
$\pi^{-1}(U_k)=\bigcup_{i=0}^{N_k}\widetilde G_{ik}$ and that 
$\mathcal O_{\widetilde X}(\Sigma)$ is trivial on $\widetilde G_{ik}$. 
We denote the transition functions of 
$\mathcal O_{\widetilde X}(\Sigma)$ by $g_{i_1k_1,i_2k_2}\in
\mathcal O(\widetilde G_{i_1k_1}\cap\widetilde G_{i_2k_2})$. 
Hence 
$$\psi_{i_1k}=\psi_{i_2k}+\log|g_{i_1k,i_2k}|\;,\;\;g_{i_1k,i_2k}
=\widetilde g_{i_1i_2,k}\,.$$

\par We now construct a smooth Hermitian metric on 
$\mathcal O_{\widetilde X}(\Sigma)$ with weights $s_{ik}$ on 
$\widetilde G_{ik}$. Let $\{\rho_k\}_{1\leq k\leq N}$ be 
a smooth partition of unity on $X$ so that 
$\supp\rho_k\subset U_k$. For a fixed $k_1\in\{1,\dots,N\}$
let $\varphi^{ik}_{k_1}$ be a function defined on 
$\pi^{-1}(U_{k_1})\cap\wi G_{ik}$ as follows: if
$x\in\wi G_{i_1k_1}$ then 
$$\varphi^{ik}_{k_1}(x)=\psi_{i_1k_1}(x)+\log|g_{ik,i_1k_1}(x)|\,.$$
Note that $\varphi^{ik}_{k_1}$ is well defined and smooth on 
$\pi^{-1}(U_{k_1})\cap\wi G_{ik}$. Indeed, if 
$x\in\wi G_{i_1k_1}\cap\wi G_{i'_1k_1}\cap\wi G_{ik}$ then 
$g_{i'_1k_1,i_1k_1}\cdot g_{i_1k_1,ik}\cdot g_{ik,i'_1k_1}=1$ so 
\begin{eqnarray*}
\psi_{i'_1k_1}(x)+\log|g_{ik,i'_1k_1}(x)|&=&\psi_{i_1k_1}(x)
+\log|g_{i'_1k_1,i_1k_1}(x)|+\log|g_{ik,i'_1k_1}(x)|\\
&=&\psi_{i_1k_1}(x)+\log|g_{ik,i_1k_1}(x)|\,.
\end{eqnarray*}
Next we define 
$$s_{ik}=\sum_{k_1=1}^N(\rho_{k_1}\circ\pi)
\varphi^{ik}_{k_1}\,\text{ on }\,\widetilde G_{ik}\,.$$
We claim that $s_{ik}=s_{i'k'}+\log|g_{ik,i'k'}|$ on 
$\wi G_{ik}\cap\wi G_{i'k'}$, so $\{s_{ik}\}$  defines  
a smooth Hermitian metric on $\mathcal O_{\widetilde X}(\Sigma)$. 
For this we show that if $k_1\in\{1,\dots,N\}$ then 
$$(\rho_{k_1}\circ\pi)\varphi^{ik}_{k_1}
=(\rho_{k_1}\circ\pi)(\varphi^{i'k'}_{k_1}+\log|g_{ik,i'k'}|).$$ 
Let $x\in\pi^{-1}(U_{k_1})$ and assume 
$x\in\wi G_{i_1k_1}\cap\wi G_{ik}\cap\wi G_{i'k'}$. Since 
$g_{i'k',i_1k_1}\cdot g_{i_1k_1,ik}\cdot g_{ik,i'k'}=1$ we obtain
\begin{eqnarray*}
\varphi^{ik}_{k_1}(x)&=&\psi_{i_1k_1}(x)+\log|g_{ik,i_1k_1}(x)|
=\psi_{i_1k_1}(x)+\log|g_{i'k',i_1k_1}(x)|+\log|g_{ik,i'k'}(x)|\\
&=&\varphi^{i'k'}_{k_1}(x)+\log|g_{ik,i'k'}(x)|\,,
\end{eqnarray*}
which proves our claim. 

\par We finally show that the desired metric $\theta$ on $F$ 
is the metric defined by 
the weights $\{-s_{ik}\}$. By a standard compactness argument 
it suffices to prove that for every $x\in\wi G_{ik}$ there exists 
$C_x>0$ such that for $C>C_x$ the function 
$C\eta_k\circ\pi-s_{ik}$ is strictly psh at $x$. We write $T_x\wi X=E_x\oplus F_x$, 
where $E_x=\ker d\pi(x)$. Note that the Levi form (see e.g. 
\cite[p. \ 228]{Ho2} for the definition) of the function $\eta_k\circ\pi$ 
is given by 
$$\mathcal L(\eta_k\circ\pi)(x)(t,t')
=\mathcal L\eta_k(\pi(x))(d\pi(x)(t),d\pi(x)(t'))\,,$$
so 
\begin{equation}\label{e:Levi1}
\mathcal L(\eta_k\circ\pi)(x)(t,t')=0\;,\;\;
\forall\,t\in T_x\wi X,\,t'\in E_x\,.
\end{equation}
Moreover, since $\eta_k$ is strictly psh at $\pi(x)\in B_k$ and 
$d\pi(x)$ is injective on $F_x$ we deduce that 
\begin{equation}\label{e:Levi2}
\mathcal L(\eta_k\circ\pi)(x)(t,t)>0\;,\;\;\forall\,t\in 
F_x\setminus\{0\}\,.
\end{equation}
The formula of $s_{ik}$ implies that for each $t\in E_x$ 
$$-\mathcal L s_{ik}(x)(t,t)=-\sum_{k_1=1}^N\rho_{k_1}(\pi(x))
\mathcal  L\psi_{i_1k_1}(x)(t,t)\geq0\,,$$
since each function $-\psi_{i_1k_1}$ is psh on $\wi G_{i_1k_1}$.
If $\rho_{k_1}(\pi(x))>0$ we may assume that 
$x\in\wi G_{i_1k_1}\cap\wi G_{ik}$ and we will show that 
$-\mathcal L\psi_{i_1k_1}(x)(t,t)>0$ for all 
$t\in E_x\setminus\{0\}$. As 
$\pi^{-1}(U_{k_1})=\Gamma_{k_1}\subset U_{k_1}
\times{\mathbb P}^{N_{k_1}}\subset B_{k_1}
\times{\mathbb P}^{N_{k_1}}$, $x$ has a neighborhood
$\wi U\subset\wi X$ such that 
$\wi U\subset B_{k_1}\times G_{i_1k_1}$. Recall that on $\wi U$, 
$\psi_{i_1k_1}=\psi'_{i_1k_1}\circ\pi_{2k_1}$ and 
$\pi=\pi_{1k_1}$. We consider $\pi_{1k_1},\,\pi_{2k_1}$ 
as restrictions of the canonical projections 
$\pi_{1k_1}:B_{k_1}\times G_{i_1k_1}\longrightarrow B_{k_1}$, 
$\pi_{2k_1}:B_{k_1}\times G_{i_1k_1}\longrightarrow G_{i_1k_1}$. 
Since $t\in E_x\subset\ker d\pi_{1k_1}(x)$ and $t\neq0$
it follows that $d\pi_{2k_1}(x)(t)\neq0$. Therefore 
$$-\mathcal L\psi_{i_1k_1}(x)(t,t)
=-\mathcal L\psi'_{i_1k_1}(\pi_{2k_1}(x))(d\pi_{2k_1}(x)(t),
d\pi_{2k_1}(x)(t))>0\,,$$
as $-\psi'_{i_1k_1}$ is strictly psh on $G_{i_1k_1}$. This yields 
\begin{equation}\label{e:Levi3}
-\mathcal L s_{ik}(x)(t,t)>0\;,\;\;\forall\,t\in E_x\setminus\{0\}\,.
\end{equation}
By \eqref{e:Levi1}, \eqref{e:Levi2} and \eqref{e:Levi3} 
we conclude that there exists
$C_x>0$ such that if $C>C_x$ then 
$\mathcal L(C\eta_k\circ\pi-s_{ik})(x)(t,t)>0$ for all 
$t\in T_x\wi X\setminus\{0\}$. This finishes the proof.
\end{proof}

\par We look now at the nature of the base space $X$ as
implied by the hypotheses made on the curvature of the bundles
involved in our results. Recall that a compact irreducible complex 
space $X$ of dimension $n$ is called Moishezon if $X$ possesses 
$n$ algebraically independent meromorphic functions,
i.e.,\ if the transcendence degree of the field of meromorphic
functions on $X$ equals the complex dimension of $X$. 
Let $X'$ and $X$ be compact irreducible spaces and $h:X'\to X$ 
be a proper modification. Then $h$ induces an isomorphism of the 
fields of meromorphic functions on $X'$ and $X$, respectively, 
\cite[Theorem\,2.1.18]{MM07}, hence $X'$ is Moishezon if and only
if $X$ is Moishezon.
Moishezon \cite{M67} showed that if $X$ is a Moishezon space,
then there exists a
proper modification $h:X'\to X$, obtained by a finite number of 
blow-ups with
smooth centers, such that $X'$ is a projective algebraic manifold 
(for a proof see also \cite[Theorem\,2.2.16]{MM07}).

Lemma \ref{L:Moishezon} yields in particular the following:

\begin{Proposition}\label{P:Moishezon} 
If $(X,\omega)$ is a compact (reduced) Hermitian space endowed 
with a singular Hermitian holomorphic 
line bundle $(L,h)$ such that $c_1(L,h)\geq\varepsilon\omega$ 
for some constant $\varepsilon>0$ then 
$X$ is Moishezon.
\end{Proposition}
\begin{proof} Let $\pi:\widetilde X\longrightarrow X$ be a 
    resolution of singularities as in Lemma \ref{L:Moishezon}, 
    with exceptional divisor $\Sigma$ and with Hermitian form 
    $\Omega=C\pi^*\omega+c_1(F,\theta)$, where 
    $F=\mathcal O_{\widetilde X}(-\Sigma)$. Consider the line 
    bundles $E_p=\pi^* L^p\otimes F$ with singular Hermitian 
    metrics $\theta_p=\pi^* h^p\otimes\theta$. Then 
$$c_1(E_p,\theta_p)=p\pi^* c_1(L,h)+c_1(F,\theta)\geq 
p\varepsilon\pi^*\omega+c_1(F,\theta)\geq\Omega\,,$$
provided that $p\varepsilon\geq C$. 
Hence $\wi X$ carries a singular Hermitian holomorphic line bundle
with strictly positive curvature in the sense of currents.
By \cite{JS93} (see also \cite[Theorem\,2.3.8]{MM07}) it follows 
that $\wi X$ is Moishezon and hence $X$, too.
If $(X,\omega)$ is K\"ahler, then $\wi X$ is already projective. 
Indeed, $(\wi X,\Omega)$ is K\"ahler and Moishezon, so by 
a theorem of Moishezon is projective 
(see e.\,g.\ \cite[Theorem\,2.2.26]{MM07}). 
\end{proof}
The paper \cite{M05} (see also 
\cite[Theorems 3.4.10 and 3.4.14 ]{MM07}) 
gives an integral criterion for a  complex space with isolated 
singularities
 to be Moishezon, generalizing the criterion of Siu and Demailly 
 from the smooth case. 

We recall next the following projectivity criterion.
\begin{Proposition}[Grauert]\label{P:Grauert} 
If $(X,\omega)$ is a compact (reduced) Hermitian space endowed 
with a $\cC^2$ Hermitian holomorphic 
line bundle $(L,h)$ such that $c_1(L,h)\geq\varepsilon\omega$
 for some constant $\varepsilon>0$, then 
$L$ is ample and $X$ is projective.
\end{Proposition}
This follows from \cite[Satz\,2,\,p.343]{Gra:62}, 
see also \cite[Satz\,3,\,p.346]{Gra:62}. If $X$ is normal, 
Grauert actually shows more: if $\Omega$ is a continuous 
K\"ahler metric whose de Rham cohomology class is integral,
 then there exists a Hermitian holomorphic line bundle $(L,h)$ 
 with $\cC^2$ Hermitian metric and $c_1(L,h)=\Omega$; 
 hence $L$ is ample and $X$ is projective.  

\subsection{$L^2$-estimates for $\db$}\label{SS:db} 
The following version of Demailly's estimates for the $\db$ operator
\cite[Th\'eor\`eme 5.1]{D82} will be needed in our proofs.

\begin{Theorem}[{\cite{D82}}]\label{T:db} Let $Y$, $\dim Y=n$,
    be a complete K\"ahler manifold and let $\Omega$ be a K\"ahler
    form on $Y$ (not necessarily complete) such that its Ricci form 
    $\ric_\Omega\geq-2\pi B\Omega$ on $Y$, for some constant 
    $B>0$. Let $(L_p,h_p)$ be singular Hermitian holomorphic line 
    bundles on $Y$ such that $c_1(L_p,h_p)\geq2a_p\Omega$, 
    where $a_p\to\infty$ as $p\to\infty$, and fix $p_0$ such that 
    $a_p\geq B$ for all $p>p_0$. If $p>p_0$ and 
    $g\in L^2_{0,1}(Y,L_p,loc)$ verifies $\db g=0$ and 
    $\int_Y|g|^2_{h_p}\,\Omega^n<\infty$ then there exists 
    $u\in L^2_{0,0}(Y,L_p,loc)$ such that $\db u=g$ and 
    $\int_Y|u|^2_{h_p}\,\Omega^n\leq\frac{1}{a_p}\,
    \int_Y|g|^2_{h_p}\,\Omega^n$. 
\end{Theorem}

\begin{proof} We write $L_p=F_p\otimes K_Y$, where 
    $F_p=L_p\otimes K^{-1}_Y$, and 
the canonical line bundle $K_Y$ is endowed with the metric 
$h^{K_Y}$ induced by $\Omega$. 
If $\theta_p=h_p\otimes h^{K^{-1}_Y}$ is the metric induced on 
$F_p$ then 
$$c_1(F_p,\theta_p)=c_1(L_p,h_p)-c_1(K_Y,h^{K_Y})
=c_1(L_p,h_p)+\frac{1}{2\pi}\,\ric_\Omega\geq(2a_p-B)\Omega
\geq a_p\Omega$$
for $p>p_0$. The theorem follows by using the isometries 
$L^2_{0,j}(Y,L_p,loc)\cong L^2_{n,j}(Y,F_p,loc)$, $j=0,1$, and 
applying \cite[Th\'eor\`eme 5.1]{D82}
(see also \cite[Corollaries 4.2 and 4.3]{CM13}).
\end{proof}

\subsection{Special weights of Hermitian metrics on reference 
covers}\label{SS:refcov} Let \((X,\omega)\) be a compact K\"ahler 
manifold of dimension $n$. Let $(U,z)$, $z=(z_1,\ldots,z_n)$, 
be local coordinates centered at a point $x\in X$. For $r>0$ and 
$y\in U$ we denote by
\[\Delta^n(y,r)=\{z\in U: |z_j-y_j|\leq r,\:j=1,\ldots,n\}\]
the (closed) polydisk of polyradius $(r,\ldots,r)$ centered at $y$.
The coordinates $(U,z)$ are called K\"ahler at $y\in U$ if 
\begin{equation}
\omega_z=\frac i2\sum_{j=1}^{n}dz_j\wedge d\overline{z}_j
+O(|z-y|^2)\:\:\text{on \(U\)}.
\end{equation}

\begin{Definition}\label{D:refcov}
A \emph{reference cover} of \(X\) consists of the following data: 
for $j=1,\ldots,N$, a set of points \(x_j\in X\) and 
\begin{enumerate}
  \item  Stein open simply connected coordinate neighborhoods 
  \((U_j,w^{(j)})\) centered at \(x_j\equiv0\),
  \item  \(R_j>0\) such that \(\Delta^n(x_j,2R_j)\Subset U_j\) 
  and for every \(y\in\Delta^n(x_j,2R_j)\) there exist coordinates 
  on \(U_j\) which are K\"ahler at \(y\),
  \item \(X=\bigcup_{j=1}^N\Delta^n(x_j,R_j)\). 
\end{enumerate}
Given the reference cover as above we set \(R=\min R_j\).
\end{Definition}

\par We can construct a reference cover as follows:
For \(x\in X\) fix a Stein open simply connected neighborhood 
\(U\) of \(x\equiv0\in\C^{n}\) and fix \(R>0\) such that 
the polydisk \(\Delta^n(x,2R)\Subset U\) and for every
\(y\in\Delta^n(x,2R)\) there exist coordinates \((U,z)\) which
are K\"ahler at \(y\). By compactness there exist
\(x_1,\ldots,x_N\in X\) such that the above conditions are fulfilled.

On $U_j$ we consider the partial derivatives $D^\alpha_w$ of order $|\alpha|$, $\alpha\in\N^{2n}$, corresponding to the real coordinates associated to $w=w^{(j)}$. For a function $\varphi\in\cC^k(U_j)$ we set 
\begin{equation}\label{bk:0.1}
\|\varphi\|_{k}=\|\varphi\|_{k,w}
=\sup\big\{|D^\alpha_w\varphi(w)|:\,w\in\Delta^n(x_j,2R_j),
|\alpha|\leq k\big\}.
\end{equation}
Let \((L,h)\) be a Hermitian holomorphic line bundle on \(X\), 
where 
the metric $h$ is of class $\cC^\ell$. Note that \(L|_{U_j}\) is 
trivial. For $k\leq\ell$ set
\begin{equation}\label{bk:0.2}
\begin{split}
\|h\|_{k,U_j}&=\inf\big\{\|\varphi_j\|_k:\,\varphi_j
\in \cC^\ell(U_j)\text{ is a weight of $h$ on $U_j$}\big\},\\
\|h\|_k&=\max\big\{1,\|h\|_{k,U_j}:\,1\leq j\leq N\big\}.
\end{split}
\end{equation}
Recall that \(\varphi_j\) is a weight of \(h\) on \(U_j\) if there
exists a holomorphic frame \(e_j\) of \(L\) on \(U_j\) 
such that \(|e_j|_h=e^{-\varphi_j}\).

\begin{Lemma}\label{L:rc}
There exists \(C>1\) (depending on the reference cover) with 
the following property: 
Given any Hermitian line bundle \((L,h)\) on \(X\), where $h$ is of class $\cC^3$, any 
\(j\in\{1,\ldots,N\}\) 
and any \(x\in\Delta^n(x_j,R_j)\) there exist coordinates 
\(z=(z_1,\ldots,z_n)\) on \(\Delta^n(x,R)\) which are centered 
at \(x\equiv0\) and K\"ahler coordinates for \(x\) such that

\par (i) $dm\leq(1+Cr^2)\,\omega^n/n!$ and 
$\omega^n/n!\leq(1+Cr^2)\,dm$ hold
 on \(\Delta^n(x,r)\) for any \(r<R\) where \(dm=dm(z)\)
  is the Euclidean volume relative to the coordinates $z$\,,
  
\par (ii) \((L,h)\) has a weight \(\varphi\) on \(\Delta^n(x,R)\) 
with \(\varphi(z)=\sum_{j=1}^n\lambda_j|z_j|^2
+\widetilde{\varphi}(z)\), where \(\lambda_j\in\R\) and 
\(|\widetilde{\varphi}(z)|\leq C\|h\|_3|z|^3\) for 
\(z\in\Delta^n(x,R)\). 
\end{Lemma}

\begin{proof}
By the properties of a reference cover there exist coordinates 
$z$ on \(U_j\) which are K\"ahler for \(x\in\Delta^n(x_j,R_j)\) so 
$\omega=\frac i2\sum_{l=1}^{n}dz_l\wedge d\overline{z}_l
+O(|z-x|^2)$ and $(i)$ holds with a constant \(C_j\) uniform for
\(x\in\Delta^n(x_j,R_j)\). Let \(e_j\) be a frame of \(L\) on 
\(U_j\) and \(\varphi'\) be a weight of \(h\) on \(U_j\) with
$|e_j|_h=e^{-\varphi'}$ and $\|\varphi'\|_{3,w}\leq2\|h\|_3$, 
cf.\ \eqref{bk:0.1}-\eqref{bk:0.2}. By translation we may assume 
\(x=0\) and write 
$\varphi^\prime (z) = \operatorname{Re} f(z)+ \varphi^\prime_2(z) 
+ \varphi^\prime_3 (z)$, where $f(z)$ is a holomorphic polynomial 
of degree $\leq 2$ in $z$, 
$\varphi^\prime_2(z)= \sum_{k,l=1}^n \mu_{kl} z_k \bar{z}_l$,
and $\operatorname{Re} f(z) + \varphi^\prime_2(z)$ is the 
Taylor polynomial of order $2$ of $\varphi'$ at $0$. Note that
$\| \varphi^\prime \|_{3,z} 
\leq C^\prime_j \|\varphi^\prime\|_{3,w} 
\leq 2C^\prime_j \|h\|_3$, where $\|\varphi^\prime\|_{3,z}$ 
is the sup norm on $\Delta^n(x,R)$ of the derivatives of order 
$3$ of $\varphi^\prime$ in the coordinates $z$ and $C^\prime_j$ 
is a constant uniform for $x \in \Delta^n(x_j, R_j)$. This follows 
from the fact that $z,w$ are coordinates on
$U_j \Supset \Delta^n(x_j, 2R_j)$. We conclude that
$|\varphi^\prime_3(z)| \leq 2C_j^\prime \|h\|_3 |z|^3$ for
$z \in \Delta^n(x,R)$. 

\par Consider the frame $\widetilde{e}_j = e^fe_j$ of $L$ on 
$U_j$. Then $|\widetilde{e}_j|_h 
= e^{\operatorname{Re} f - \varphi^\prime} = e^{-\varphi}$,
so $\varphi(z) := \varphi^\prime_2(z) + \varphi_3^\prime(z)$
is a weight of $h$ on $U_j$. By a unitary change of coordinates 
we may assume that $\varphi (z) = \sum_{l=1}^n \lambda_l |z_l|^2 
+ \widetilde{\varphi}(z)$. In these new coordinates $\omega^n/n !$ 
and $\widetilde{\varphi}(z)$ satisfy the desired estimates with 
a constant $C_j$ uniform for 
$x \in \Delta^n(x_j, R_j)$. Finally we let 
$C = \max \limits_{1\leq j\leq N} C_j$.
\end{proof}

\section{Proofs of Theorems \ref{T:mt1}, \ref{T:mt2} and 
\ref{T:mt3}}\label{S:pfmt}

\par We use here the notations introduced in Section \ref{S:prelim}
and we start 
with two lemmas that will be needed in the proof of 
Theorem \ref{T:mt1}:

\begin{Lemma}\label{L:divm} Let $D$ be a divisor in 
    a complex manifold $Y$ and $\theta$ be a smooth 
    Hermitian metric on $\mathcal O(-D)$ over $Y$ with weight 
    $\varphi$ over $Y\setminus D$. Then 
    $\displaystyle \lim_{y\to x,\,y\in Y\setminus D}\varphi(y)
    =-\infty$ for every $x\in D$.
\end{Lemma}

\begin{proof} Let $U_\alpha$ be a neighborhood of $x$ where 
$D$ has defining function $f_\alpha\in\mathcal O(U_\alpha)$ 
 and $\theta$ has weight $\varphi_\alpha\in \cC^\infty(U_\alpha)$. 
The transition function of $\mathcal O(-D)$ on 
$U_\alpha\cap(Y\setminus D)$ is $g=1/f_\alpha$, so 
$\varphi_\alpha=\varphi+\log|g|$ and 
$\varphi=\varphi_\alpha+\log|f_\alpha|$. 
\end{proof}

\begin{Lemma}\label{L:tildehp} Let $(X,\omega),\,(L_p,h_p)$ verify 
assumptions (A)-(B), and let $F=\mathcal O_{\wi X}(-\Sigma)$, 
$\theta$, $\Omega=C\pi^*\omega+c_1(F,\theta)$ be as in 
Lemma \ref{L:Moishezon}. Then there exist $\alpha\in(0,1)$, 
$b_p\in\mathbb N$, 
and singular Hermitian metrics $\wi h_p$ on 
$\pi^* L_p\vert_{\wi X\setminus\Sigma}$ 
such that $a_p\geq Cb_p\,$, $b_p\to\infty$ and $b_p/A_p\to0$ 
as $p\to\infty$, $\wi h_p\geq\alpha^{b_p}\pi^* h_p$ and 
 $c_1(\pi^* L_p,\wi h_p)\geq b_p\Omega$ on
 $\wi X\setminus\Sigma$. 
 Moreover, for every open relatively compact subset $\wi U$ 
 of $\wi X\setminus\Sigma$ there exists a constant 
 $\beta_{\wi U}>1$ such that 
 $\wi h_p\leq\beta_{\wi U}^{b_p}\pi^* h_p$ on $\wi U$.  
\end{Lemma}

\begin{proof} If $h',\,h''$ are singular Hermitian metrics on some 
    holomorphic line bundle $G$ then by $h'\geq ch''$ we mean that
    $|e|^2_{h'}\geq c|e|^2_{h''}$ for any $e\in G$. Consider the line 
    bundles $E_p=\pi^* L_p\otimes F^{b_p}$ with metrics 
    $\theta_p=\pi^* h_p\otimes\theta^{b_p}$, where 
    $b_p\in\mathbb N$. Then 
$$c_1(E_p,\theta_p)=\pi^* c_1(L_p,h_p)+b_pc_1(F,\theta)\geq 
a_p\pi^*\omega
+b_pc_1(F,\theta)\geq b_p\Omega\,,$$
provided that $a_p\geq Cb_p$. Since $F$ is trivial on 
$\wi X\setminus\Sigma$ 
we have $\pi^* L_p\vert_{\wi X\setminus\Sigma}\cong E_p
\vert_{\wi X\setminus\Sigma}$ 
and we can find a smooth weight $\varphi$ of $\theta$ on 
$\wi X\setminus\Sigma$ by setting $|f|_\theta^2=e^{-2\varphi}$,
where $f$ is a holomorphic frame of $F$ 
on $\wi X\setminus\Sigma$. Let $\wi h_p$ be the metric of
 $\pi^* L_p\vert_{\wi X\setminus\Sigma}$ defined by 
 $\wi h_p=e^{-2b_p\varphi}\pi^* h_p$. Then 
 $c_1(\pi^* L_p,\wi h_p)=c_1(E_p,\theta_p)\geq b_p\Omega$ on 
 $\wi X\setminus\Sigma$. Since 
 $\varphi\in \cC^\infty(\wi X\setminus\Sigma)$ 
 and by Lemma \ref{L:divm} $\varphi(y)\to-\infty$ as $y\to\Sigma$,
  it follows that there exists $\alpha\in(0,1)$ such that 
  $e^{-2\varphi}\geq\alpha$ 
  on $\wi X\setminus\Sigma$. Moreover, if $\wi U$ is an open 
relatively compact subset of $\wi X\setminus\Sigma$, there exists 
$\beta_{\wi U}>1$ such that $e^{-2\varphi}\leq\beta_{\wi U}$ 
on $\wi U$. These imply that $\wi h_p\geq\alpha^{b_p}\pi^* h_p$ 
on $\wi X\setminus\Sigma$ and 
  $\wi h_p\leq\beta_{\wi U}^{b_p}\pi^* h_p$ on $\wi U$. 
The lemma follows if we choose $b_p\in\mathbb N$ such that 
$a_p\geq Cb_p$, $b_p\to\infty$ 
  and $b_p/A_p\to0$ as $p\to\infty$. 
\end{proof}

\smallskip

\begin{proof}[Proof of Theorem \ref{T:mt1}] Note that $(ii)$
    follows at once from $(i)$ by using \eqref{e:BFS2}. 
    The proof of $(i)$ will be done in two steps. 

\smallskip

\par {\em Step 1.} We show here that 
$\frac{1}{A_p}\,\log P_p\to0$ as $p\to\infty$ in 
$L^1_{loc}(X_{\rm reg},\omega^n)$. Fix $x\in X_{\rm reg}$, 
$W\Subset X_{\rm reg}$ a contractible Stein coordinate 
neighborhood of $x$, $r_0>0$ such that the (closed) ball 
$V:=B(x,2r_0)\subset W$, and set $U:=B(x,r_0)$. 

Note that on a Stein manifold $M$ we have $H^1(M,\cO^*)\cong H^2(M,\Z)$
due to Cartan's theorem B (see e.\,g.\ \cite[p.\,201]{Ho1}), thus
any holomorphic line bundle $L$ over a Stein contractible
manifold $\Omega$ is holomorphically trivial
(this is due to \cite{Oka}, and is of course a special case of the
Oka-Grauert principle). 
Thus there exist local holomorphic frames $e'_p:W\longrightarrow L_p$, for all $p$. 
Let $\psi'_p$ be the corresponding psh weights of $h_p$ on $W$, 
$|e'_p|^2_{h_p}=e^{-2\psi'_p}$. The sequence of currents 
$\{\frac{1}{A_p}\,c_1(L_p,h_p)\}$ has uniformly bounded mass, 
so it follows from \cite[Proposition A.16]{DS10} 
(see also \cite{DNS08}) that there exist psh functions $\psi_p$ 
on $\inte V$ such that $dd^c\psi_p=c_1(L_p,h_p)$ and 
the sequence $\{\frac{1}{A_p}\,\psi_p\}$ is bounded, hence 
relatively compact \cite[Theorem 3.2.12]{Ho2}, in 
$L^1_{loc}(\inte V,\omega^n)$. Since $\psi'_p-\psi_p$ is
pluriharmonic, we have $\psi'_p-\psi_p=\re f_p$ for some function
$f_p\in\mathcal O(\inte V)$. Considering the local frames 
$e_p=e^{f_p}e'_p$ of $L_p\vert_{\inte V}$, we obtain 
$$|e_p|^2_{h_p}=e^{2\re f_p}|e'_p|^2_{h_p}=e^{-2\psi_p},$$
i.e.,\ $\psi_p$ is the psh weight of $h_p$ relative to the frame 
$e_p$. 

\par Let $\{b_p\}$ be as in Lemma \ref{L:tildehp}. We prove that 
there exist 
$C'>1$ and $p_0\in\mathbb N$ such that 
\begin{equation}\label{e:Bkf}
-\frac{b_p\log C'}{A_p}\leq\frac{\log P_p(z)}{A_p}
\leq\frac{\log(C'r^{-2n})}{A_p}+\frac{2}{A_p}\,\left(\max_{B(z,r)}
\psi_p-\psi_p(z)\right)
\end{equation}
holds for all $p>p_0$, $0<r<r_0$, and $z\in U$ with 
$\psi_p(z)>-\infty$.
 The upper bound in \eqref{e:Bkf} is proved exactly as the 
 corresponding upper 
 bound from the proof of \cite[Theorem 5.1]{CM11}. 
 For the lower bound, 
 we show that there exist $c\in(0,1)$ and $p_0\in\mathbb N$ 
 such that if $p>p_0$ and $z\in U$, $\psi_p(z)>-\infty$, 
 then there exists a section
  $S_{z,p}\in H^0_{(2)}(X,L_p)$ verifying $S_{z,p}(z)\neq0$ and 
\begin{equation}\label{e:lest}
c^{b_p}\|S_{z,p}\|_p^2\leq|S_{z,p}(z)|^2_{h_p}\,.
\end{equation}
This implies that 
$$\frac{1}{A_p}\,\log P_p(z)=\frac{1}{A_p}\,\max_{\|S\|_p=1}
\log|S(z)|^2_{h_p}\geq\frac{b_p\log c}{A_p}\,.$$

\par To prove \eqref{e:lest} we work on $\wi X\setminus\Sigma$ 
and recall that $\pi:\wi X\setminus\Sigma\longrightarrow 
X_{\rm reg}$ is a biholomorphism. Let $\Omega\geq\pi^*\omega$
be the K\"ahler form on $\wi X$ constructed in
Lemma \ref{L:Moishezon}, and $b_p,\,\wi h_p$ be as in 
Lemma \ref{L:tildehp}. Then $b_p\to\infty$ and 
$$c_1(\pi^* L_p\vert_{\wi X\setminus\Sigma},\wi h_p)
\geq b_p\Omega\,\text{ on }\,\wi X\setminus\Sigma.$$
Since $\Omega$ is a K\"ahler form on $\wi X$ we have
$\ric_\Omega\geq-2\pi B\Omega$ on $\wi X$, for some 
constant $B>0$. Moreover, since $\wi X$ is a compact K\"ahler 
manifold, $\wi X\setminus\Sigma$ has a complete K\"ahler metric 
(see \cite{D82}, \cite{O87}). Repeating the argument in the proof 
of \cite[Theorem 5.1]{CM11} 
(see also \cite[Theorems 4.2 and 4.3]{CM13}) one applies 
the Ohsawa-Takegoshi extension theorem \cite{OT87} and then
solves a suitable $\db$-equation using Theorem \ref{T:db}, 
as in \cite[Proposition 3.1]{D92}, \cite[Section 9]{D93b}, 
to show the following: there exist $C''>1$, $p_0\in\mathbb N$, 
such that if $p>p_0$ and $\wi z\in\pi^{-1}(U)$, 
$\psi_p\circ\pi(\wi z)>-\infty$, then there exists a section 
$\wi S\in H^0(\wi X\setminus\Sigma,\pi^* L_p)$ verifying 
$\wi S(\wi z)\neq0$ and 
$$\int_{\wi X\setminus\Sigma}|\wi S|^2_{\wi h_p}\,\frac{\Omega^n}{n!}\leq C''|\wi S(\wi z)|^2_{\wi h_p}\,.$$
It is important to recall here that the weight of $\wi h_p$ near
$\wi z$ is the sum of $\psi_p\circ\pi$ and a smooth function. 
By Lemma \ref{L:tildehp} we have for all $p$ that 
$\wi h_p\geq\alpha^{b_p}\pi^* h_p$ on $\wi X\setminus\Sigma$ 
and $\wi h_p\leq\beta^{b_p}\pi^* h_p$ on $\pi^{-1}(U)$, 
for some constant $\beta>1$. Since $\Omega\geq\pi^*\omega$
these imply that 
$$\alpha^{b_p}\int_{\wi X\setminus\Sigma}|\wi S|^2_{\pi^* h_p}\,
\frac{\pi^*\omega^n}{n!}\leq C''\beta^{b_p}|\wi S(\wi z)|^2_{\pi^* h_p}\,.$$
Fix a constant $c\in(0,1)$ with 
$C''c^{b_p}\leq(\alpha/\beta)^{b_p}$ for all $p$ and let $S_{z,p}$, 
where $z=\pi(\wi z)$, be the section of $L_p\vert_{X_{\rm reg}}$
induced by $\wi S$. Since $X$ is normal it follows that $S_{z,p}$ 
extends to a holomorphic section of $L_p$ on $X$. Moreover, 
$S_{z,p}(z)\neq0$ and \eqref{e:lest} holds for $S_{z,p}$ and $c$. 
The proof of \eqref{e:Bkf} is now complete. 

\par To conclude Step 1, it suffices to show that every 
subsequence of $\{\frac{1}{A_p}\,\log P_p\}$ has a subsequence 
convergent to $0$ in $L^1(U,\omega^n)$. 
Without loss of generality, we prove that 
$\{\frac{1}{A_p}\,\log P_p\}$ has a subsequence convergent to 
$0$ in $L^1(U,\omega^n)$. Since $\{\frac{1}{A_p}\,\psi_p\}$ 
is locally uniformly upper bounded in $\inte V$ and relatively 
compact in $L^1_{loc}(\inte V,\omega^n)$, there exists a 
subsequence $\{\psi_{p_j}\}$ so that 
$$\frac{1}{A_{p_j}}\,\psi_{p_j}\to\psi
=\left(\limsup\frac{1}{A_{p_j}}\,\psi_{p_j}\right)^*$$
in $L^1_{loc}(\inte V,\omega^n)$ and a.e. on $\inte V$, where
$\psi\in PSH(\inte V)$. Moreover, by the Hartogs lemma, 
$$\limsup\frac{1}{A_{p_j}}\,\max_{B(z,r)}\psi_{p_j}
\leq\max_{B(z,r)}\psi\,,$$
for each $z\in U$ and $r<r_0$
(see e.g., \cite[Theorem 3.2.13]{Ho2}). Letting $p_j\to\infty$
in \eqref{e:Bkf} we get, since $b_p/A_p\to0$, that 
$$0\leq\liminf\frac{\log P_{p_j}(z)}{A_{p_j}}
\leq\limsup\frac{\log P_{p_j}(z)}{A_{p_j}}
\leq2\left(\max_{B(z,r)}\psi-\psi(z)\right)$$
for a.e. $z\in U$ and every $r<r_0$. Letting $r\searrow0$ 
and using the upper semicontinuity of $\psi$ we deduce that 
$\frac{1}{A_{p_j}}\,\log P_{p_j}\to0$ a.e. on $U$.
Since $\{\frac{1}{A_p}\,\psi_p\}$ is locally uniformly upper 
bounded in $\inte V$ it follows by \eqref{e:Bkf} that there exists 
a constant $C''>0$ such that 
$$\left|\frac{1}{A_{p_j}}\,\log P_{p_j}\right|
\leq C''-\frac{2}{A_{p_j}}\,\psi_{p_j}\,\text{ a.e. on }\, U\,.$$
As $\psi_{p_j},\,\psi\in L^1(U,\omega^n)$, $\frac{1}{A_{p_j}}\,
\psi_{p_j}\to\psi$ a.e. on $U$ and in $L^1(U,\omega^n)$, and 
since $\frac{1}{A_{p_j}}\,\log P_{p_j}\to0$ a.e. on $U$, 
the generalized Lebesgue dominated convergence theorem
implies that $\frac{1}{A_{p_j}}\,\log P_{p_j}\to0$ in 
$L^1(U,\omega^n)$. 

\smallskip

\par {\em Step 2.} To complete the proof of $(i)$ we show here 
that there exists a compact set $K\subset X$ such that 
$X_{\rm sing}\subset\inte K$ and 
$$\frac{1}{A_p}\,\int_K|\log P_p|\,\omega^n\to0\,
\text{ as }\,p\to\infty\,.$$
Let $H^0_{(2)}(\wi X,\pi^* L_p)$ be the Bergman spaces from 
Lemma \ref{L:bPp}. We note that there exists $M>0$ such that 
\begin{equation}\label{e:domin}
\int_{\wi X}c_1(\pi^* L_p,\pi^* h_p)\wedge\Omega^{n-1}
\leq MA_p\;,\;\;\forall\,p\geq1\,.
\end{equation}
Indeed, since $c_1(L_p,h_p),\,T_0\in\cT(X)$ and 
$c_1(L_p,h_p)\leq A_pT_0$, we have 
$c_1(\pi^* L_p,\pi^* h_p)=\pi^* c_1(L_p,h_p)\leq A_p\pi^* T_0$. 
This yields \eqref{e:domin} with 
$M=\int_{\wi X}\pi^* T_0\wedge\Omega^{n-1}$. 

\par We fix now $y\in\Sigma$ and $\wi W$ an open neighborhood 
of $y$ in $\wi X$ biholomorphic to a ball in ${\mathbb C}^n$. 
Using \eqref{e:domin} and repeating an argument from Step 1, 
we find holomorphic frames $\wi e_p$ of $\pi^* L_p\vert_{\wi W}$
for which the corresponding psh weights $\wi\psi_p$ of 
$\pi^* h_p$ are such that $\{\frac{1}{A_p}\,\wi\psi_p\}$
is bounded in $L^1_{loc}(\wi W,\Omega^n)$, hence locally 
uniformly upper bounded on $\wi W$ and relatively compact in
$L^1_{loc}(\wi W,\Omega^n)$. If $\{\wi S^p_j:\,1\leq j\leq d_p\}$
is an orthonormal basis of $H^0_{(2)}(\wi X,\pi^* L_p)$ we write
$\wi S^p_j=\wi s^p_j\wi e_p$ and let 
$\wi v_p=\frac{1}{2}\,\log(\sum_{j=1}^{d_p}|\wi s^p_j|^2)
\in PSH(\wi W)$. By Lemma \ref{L:bPp}, $P_p\circ\pi$
is the Bergman kernel function of $H^0_{(2)}(\wi X,\pi^* L_p)$,
hence 
$$\frac{1}{A_p}\,\wi v_p-\frac{1}{A_p}\,
\wi\psi_p=\frac{1}{2A_p}\,\log P_p\circ\pi\,.$$
We claim that $\frac{1}{A_p}\,\log P_p\circ\pi\to0$ in 
$L^1_{loc}(\wi W,\Omega^n)$. 
As in Step 1, it suffices to produce a subsequence with 
this property. 
Since $\{\frac{1}{A_p}\,\wi\psi_p\}$ is relatively compact in 
$L^1_{loc}(\wi W,\Omega^n)$ there is a subsequence 
$\{\frac{1}{A_{p_j}}\,\wi\psi_{p_j}\}$ 
convergent in $L^1_{loc}(\wi W,\Omega^n)$ to a psh function 
$\wi\psi$ on $\wi W$. 
By Step 1, $\frac{1}{A_p}\,\log P_p\to0$ as $p\to\infty$ in 
$L^1_{loc}(X_{\rm reg},\omega^n)$, so as 
$\pi:\wi X\setminus\Sigma\longrightarrow X_{\rm reg}$ is
biholomorphic, 
$\frac{1}{A_p}\,\log P_p\circ\pi\to0$ in
$L^1_{loc}(\wi W\setminus\Sigma,\Omega^n)$. 
Thus $\frac{1}{A_{p_j}}\,\wi v_{p_j}\to\wi\psi$ in 
$L^1_{loc}(\wi W\setminus\Sigma,\Omega^n)$. By the argument in 
the proof of \cite[ Theorem 1.1 (i)]{CM13}, we see that 
$\{\frac{1}{A_{p_j}}\,\wi v_{p_j}\}$ is locally uniformly upper 
bounded in $\wi W$ 
and it converges to $\wi\psi$ in $L^1_{loc}(\wi W,\Omega^n)$. 
Hence $\frac{1}{A_{p_j}}\,\log P_{p_j}\circ\pi\to0$ in 
$L^1_{loc}(\wi W,\Omega^n)$, which proves our claim. 

\par Since $y\in\Sigma$ was arbitrary and $\Sigma$ is compact 
we can find an open set $\wi U\supset\Sigma$ so that
$\frac{1}{A_p}\,\log P_p\circ\pi\to0$ in $L^1(\wi U,\Omega^n)$. 
Then we fix a compact set $K\subset X$ such that 
$X_{\rm sing}\subset\inte K$ and $\pi^{-1}(K)\subset\wi U$. 
As $\pi:\wi X\setminus\Sigma\longrightarrow X_{\rm reg}$ is 
biholomorphic and $\pi^*\omega\leq\Omega$ we have 
\begin{eqnarray*}
\frac{1}{A_p}\,\int_K|\log P_p|\,\omega^n
&=&\frac{1}{A_p}\,\int_{K\cap X_{\rm reg}}|\log P_p|\,\omega^n
=\frac{1}{A_p}\,\int_{\pi^{-1}(K)\setminus\Sigma}
|\log P_p\circ\pi|\,\pi^*\omega^n\\
&\leq&\frac{1}{A_p}\,\int_{\pi^{-1}(K)}|\log P_p\circ\pi|\,
\Omega^n\to0\,\text{ as }\,p\to\infty\,.
\end{eqnarray*}
This concludes Step 2 and finishes the proof of 
Theorem \ref{T:mt1}.
\end{proof}


\par Theorem \ref{T:mt1} holds under the weaker hypotheses 
obtained if one replaces the domination assumption \eqref{e:domin0}
by \eqref{e:domin}, where $\pi:\widetilde{X}\to X$ is the resolution
of singularities fixed in Section \ref{S:prelim}. The proof goes 
through without change.
Due to the presence of singularities of $X$, it is not clear
whether \eqref{e:domin} holds true without 
the domination condition \eqref{e:domin0}, i.e.,\ if the mass 
of the pull-back currents $\pi^*c_1(L_p,h_p)$ on $\widetilde{X}$
is dominated by the mass of the currents $c_1(L_p,h_p)$ on $X$,
uniformly in $p$\,. 


\par\begin{proof}[Proof of Theorem \ref{T:mt2}]  We repeat the 
argument in Step 1 from the previous proof, working directly on
$(X,L_p,h_p)$ with the K\"ahler form $\omega$. 
\end{proof}



\begin{proof}[Proof of Theorem \ref{T:mt3}] We use methods 
    from \cite[Section 2]{Be03}. Let us consider a reference cover 
    of $X$ as in Definition \ref{D:refcov}. We fix $x \in X$, so 
    $x \in \Delta^n(x_j,R_j)$ for some $j$, and we pick coordinates
    $z$ centered at $x$ as in Lemma \ref{L:rc}. Let 
$$\varphi_p(z) = \varphi^\prime_p(z)+ \widetilde{\varphi}_p (z)\,,
\;\;\varphi^\prime_p(z) = \sum_{l=1}^n \lambda^p_l |z_l|^2\,,$$
be a weight of $h_p$ on $\Delta^n(x,R)$ so that 
$\widetilde{\varphi}_p$ verifies $(ii)$ in Lemma \ref{L:rc} and let
$e_p$ be a frame of $L_p$ on $U_j$ with
$|e_p|_{h_p} = e^{-\varphi_p}$. Finally, let $r_p\in(0,R/2)$
be an arbitrary number which will be specified later. 

\par We begin by estimating the norm of a section 
$S \in H^0(X,L_p)$ at $x$. Writing $S = se_p$, where 
$s \in \mathcal{O}(\Delta^n(x,R))$, we obtain by the 
sub-averaging inequality for psh functions
\[
|S(x)|^2_{h_p} = |s(0)|^2 \leq \frac{\int_{\Delta^n(0,r_p)} 
|s|^2 e^{-2 \varphi^\prime_p}\,dm}{\int_{\Delta^n(0,r_p)} 
e^{-2 \varphi^\prime_p} \,dm}\,\cdot
\]
If $C>1$ is the constant from Lemma \ref{L:rc}, we have
\begin{eqnarray*}
\int_{\Delta^n(0,r_p)}|s|^2e^{-2\varphi_p'}\,dm&
\leq&(1+Cr^2_p)\exp\!\big(2\max_{\Delta^n(0,r_p)}
\widetilde{\varphi}_p\big)\int_{\Delta^n(0,r_p)}|s|^2
e^{-2\varphi_p}\frac{\omega^n}{n!}\\ 
&\leq&(1+Cr^2_p)\exp\!\big(2C\|h_p\|_3\,r^3_p\big)\|S\|^2_p\,.
\end{eqnarray*}
Set 
$$E(r):=\int_{|\xi|\leq r}e^{-2|\xi|^2}\,dm(\xi)
=\frac{\pi}{2}\,\left(1-e^{-2r^2}\right),$$ 
where $\,dm$ is the Lebesgue measure on $\C$. Since 
$\lambda_j^p\geq a_p$ we obtain 
$$\frac{E(r_p\sqrt{a_p}\,)^n}{\lambda_1^p\ldots\lambda_n^p} 
\leq \int_{\Delta^n(0,r_p)}e^{-2\varphi'_p}\,dm
\leq\int_{\C^{n}}e^{-2\varphi_p'}\,dm 
= \frac{(\pi/2)^n}{\lambda_1^p\ldots\lambda_n^p}\,\cdot$$
Combining these estimates it follows that 
\begin{equation}\label{e:bk1}
|S(x)|_{h_p}^2\leq\frac{(1+Cr_p^2)\exp\!\big(2C\|h_p\|_3\,r_p^3
\big)}{E(r_p\sqrt{a_p})^n}\,\lambda_1^p
\ldots\lambda_n^p\,\|S\|_p^2\,.
\end{equation}
By taking the supremum in \eqref{e:bk1} over all 
$S\in H^0(X,L_p)$ with $\|S\|_p=1$ we get 
\begin{equation}\label{e:bk2}
\frac{P_p(x)}{\lambda_1^p\ldots\lambda_n^p}\leq 
\frac{(1+Cr_p^2)\exp\!\big(2C\|h_p\|_3\,r_p^3\big)}{E(r_p\sqrt{a_p})^n}\,,\quad \forall\, r_p\in(0,R/2). 
\end{equation}

\par For the lower estimate on $P_p$, let $0\leq\chi\leq1$ 
be a cut-off function on $\C^{n}$ with support in $\Delta^n(0,2)$, 
$\chi\equiv 1$ on $\Delta^n(0,1)$, and set
$\chi_p(z)=\chi(z/r_p)$. Then $F=\chi_p e_p$ is a section of 
$L_p$ and $|F(x)|_{h_p}=|e_p(x)|_{h_p}=e^{-\varphi_p(0)}=1$. 
We have  
\begin{equation}\label{e:bk3}\begin{split}
\|F\|^2_p&\leq\int_{\Delta^n(0,2r_p)}e^{-2\varphi_p}\,\frac{\omega^n}{n!}\\
&\leq (1+4Cr^2_p)\exp\!\big(16C\|h_p\|_3\,r_p^3\big)
\int_{\Delta^n(0,2r_p)}e^{-2\varphi_p'}\,dm\\
&\leq\left(\frac{\pi}{2}\right)^n\,\frac{(1+4Cr_p^2)
\exp\!\big(16C\|h_p\|_3\,r_p^3\big)}{\lambda_1^p
\ldots\lambda_n^p}\,\cdot
\end{split}
\end{equation}
Set $\alpha=\overline\partial F$. Since 
$\|\overline\partial\chi_p\|^2=\|\overline\partial\chi\|^2/r_p^2$, 
where $\|\overline\partial\chi\|$ denotes the maximum of 
$|\overline\partial\chi|$, we obtain as above 
$$\|\alpha\|_p^2
=\int_{\Delta^n(0,2r_p)}|\overline\partial\chi_p|^2
e^{-2\varphi_p}\,\frac{\omega^n}{n!}\leq\frac{\|\overline\partial\chi\|^2}
{r_p^2}\,\left(\frac{\pi}{2}\right)^n\,\frac{(1+4Cr_p^2)\exp\!\big(16C\|h_p\|_3\,r_p^3\big)}{\lambda_1^p\ldots\lambda_n^p}\,\cdot$$
Since $a_p\to\infty$ there exists $p_0\in\N$ such that for 
$p>p_0$ we can solve the $\overline\partial$--equation by 
Theorem \ref{T:db}. We get a smooth section $G$ of $L_p$ 
with $\overline\partial G=\alpha=\overline\partial F$ and 
\begin{equation}\label{e:bk4}
\|G\|_p^2\leq \frac{2}{a_p}\,\|\alpha\|_p^2
\leq \frac{2\|\overline\partial\chi\|^2}{a_pr_p^2}\,
\left(\frac{\pi}{2}\right)^n\,\frac{(1+4Cr_p^2)\exp\!\big(16C\|h_p\|_3\,r_p^3\big)}{\lambda_1^p\ldots\lambda_n^p}\,\cdot 
\end{equation}
Since $F=e_p$ is holomorphic on $\Delta^n(0,r_p)$, $G$ 
is holomorphic on $\Delta^n(0,r_p)$ as 
$\overline\partial G=\overline\partial F=0$ there. 
So the estimate \eqref{e:bk1} applies to $G$ on 
$\Delta^n(0,r_p)$ and gives  
\begin{eqnarray*}
|G(x)|_{h_p}^2&\leq&\frac{(1+Cr_p^2)
\exp\!\big(2C\|h_p\|_3\,r_p^3\big)}
{E(r_p\sqrt{a_p})^n}\lambda_1^p\ldots\lambda_n^p\|G\|_p^2\\
&\leq&\frac{2\|\overline\partial\chi\|^2}{a_pr_p^2
E(r_p\sqrt{a_p})^n}\,\left(\frac{\pi}{2}\right)^n\,(1+4Cr_p^2)^2 
\exp\!\big(18C\|h_p\|_3\,r_p^3\big).
\end{eqnarray*}
Let $S=F-G\in H^0(X,L_p)$. Then
\begin{eqnarray*}
|S(x)|_{h_p}^2&\geq&(|F(x)|_{h_p}-|G(x)|_{h_p})^2
= (1-|G(x)|_{h_p})^2\\
&\geq&\left[1-\,\left(\frac{\pi}{2}\right)^{n/2}
\frac{\sqrt2\,\|\overline\partial\chi\|(1+4Cr_p^2)}{r_p\sqrt{a_p}
E(r_p\sqrt{a_p})^{n/2}}\exp\!\big(9C\|h_p\|_3\,r_p^3\big)
\right]^2=:K_1(r_p)\,.
\end{eqnarray*}
Moreover, by \eqref{e:bk3} and \eqref{e:bk4}
$$\|S\|_p^2\leq (\|F\|_p+\|G\|_p)^2
\leq \left(\frac{\pi}{2}\right)^n\,\frac{K_2(r_p)}{\lambda_1^p
\ldots\lambda_n^p} \,,$$
where 
$$K_2(r_p)=(1+4Cr_p^2) \exp\!\big(16C\|h_p\|_3\,r_p^3\big)
\left(1+\frac{\sqrt2\,\|\overline\partial\chi\|}{r_p\sqrt{a_p}}\right)^2
\cdot$$
Therefore 
\begin{equation}
\label{e:bk5}
P_p(x)\geq \frac{|S(x)|_{h_p}^2}{\|S\|_p^2}
\geq \frac{\lambda_1^p\ldots\lambda_n^p}{(\frac{\pi}{2})^n}\frac{K_1(r_p)}{K_2(r_p)}\,\cdot
\end{equation}
Note that at $x$, $\omega_x=\frac{i}{2}\sum_{j=1}^n dz_j
\wedge d\bar{z}_j\,$, 
$c_1(L_p,h_p)_x=dd^c \varphi_p(0)
=\frac{i}{\pi}\sum_{j=1}^n\lambda_j^p dz_j\wedge d\bar{z}_j\,$,
thus 
\[\frac{c_1(L_p,h_p)_x^n}{\omega^n_x}
=\left(\frac{2}{\pi}\right)^n\lambda_1^p\ldots\lambda_n^p\,.\] 
By \eqref{e:bk2} and \eqref{e:bk5} we conclude that 
\begin{equation}
\label{e:bk6}
\frac{K_1(r_p)}{K_2(r_p)}\leq P_p(x)\frac{\omega^n_x}{c_1(L_p, h_p)_x^n}\leq K_3(r_p)
\end{equation}
holds for every $x\in X$, $r_p<R/2$ and $p>p_0$, where 
$$K_3(r_p)=\left(\frac{\pi/2}{E(r_p\sqrt{a_p})}\right)^n
(1+Cr_p^2)\exp\!\big(2C\|h_p\|_3\,r_p^3\big)\,.$$
By \eqref{e:3d} we have that 
$\varepsilon_p=\|h_p\|_3^{1/3}a_p^{-1/2}\to0$. We set
$$r_p:=\varepsilon_p^{1/3}\|h_p\|_3^{-1/3}
=\varepsilon_p^{-2/3}a_p^{-1/2}\;,\,\text{ so }\,\|h_p\|_3\,r_p^3
=\varepsilon_p\,,\;r_p\sqrt{a_p}=\varepsilon_p^{-2/3}\,.$$
As $\|h_p\|_3\geq1$, we have $r_p\leq \varepsilon_p^{1/3}$, 
thus $r_p\to0$ as $p\to\infty$. With this choice for $r_p$ we 
obtain 
$$K_3(r_p)\leq\left(\frac{\pi/2}{
E(\varepsilon_p^{-2/3})}\right)^n(1+C\varepsilon_p^{2/3})
\exp\!\big(2C\varepsilon_p\big)\leq1+C'\varepsilon_p^{2/3}\;,\;\;
\frac{K_1(r_p)}{K_2(r_p)}\geq1-C'\varepsilon_p^{2/3}\,,$$
where $C'>0$ is a constant depending only on the reference cover.
Therefore \eqref{e:Bexp00} follows from \eqref{e:bk6} and the 
proof is complete. 
\end{proof}


\begin{Remark}\label{R:Bka} Theorem \ref{T:mt3} shows that
$$\lim_{p\to\infty}P_p(x)\,\frac{\omega^n_x}{c_1(L_p, h_p)_x^n}=1\,,\:\:\text{uniformly on $X$}.$$
This is a generalization of the asymptotic expansion of the Bergman
kernel \cite{Ca99,DLM04a,HsM11,MM07,MM08,Ru98,Ti90,Z98}
for $(L_p,h_p)=(L^p,h^p)$, where $(L,h)$ is a positive line bundle 
with smooth metric $h$. Indeed, if $(L_p,h_p)=(L^p,h^p)$, we have
$a_p=p$ and $\|h_p\|_3\leq C_hp$, where $C_h$ is a constant 
depending only on $h$ and the reference cover. Hence
\[\varepsilon_p=\|h_p\|_3^{1/3}a_p^{-1/2}
\leq C_h^{1/3}p^{-1/6}\,,\]
so condition \eqref{e:3d} is fulfilled. Estimate \eqref{e:Bexp00} 
yields
\begin{equation}\label{bk:16}
\left|\frac{P_p(x)}{p^n}\, \frac{\omega^n_x}{c_1(L, h)_x^n}-1\right|\leq \frac{CC_h^{2/9}}{p^{1/9}}\,,
\end{equation}
hence
\begin{equation}\label{bk:17}
|P_p(x)-b_0(x)p^n|\leq C'p^{n-\frac{1}{9}}\,,\,
\text{ where }\,b_0(x):=\frac{c_1(L, h)_x^n}{\omega^n_x}\,\cdot
\end{equation}
By the above mentioned papers there exists $C>0$ such that  
$\big|P_p(x)-b_0(x)p^n\big|\leq C p^{n-1}$ on $X$, which gives 
a sharper estimate of $P_p(x)$ than \eqref{bk:16}-\eqref{bk:17}.
On the other hand, the method used here can handle the much more
general case of sequences of line bundles $(L_p,h_p)$ satisfying
the minimal hypotheses of Theorem \ref{T:mt3}. 
\end{Remark}


\section{Zeros of holomorphic sections and 
approximation results}\label{S:zeros}

We assume as before that $(X,\omega)$, $(L_p,h_p)\to X$ 
satisfy conditions (A) and (B).
Consider the unit sphere 
${\mathcal S}^p\subset H^0_{(2)}(X,L_p)$, 
$d_p=\dim H^0_{(2)}(X,L_p)$. We identify the unit sphere 
${\mathcal S}^p$ to the unit sphere ${\mathbf S}^{2d_p-1}$ 
in ${\mathbb C}^{d_p}$ by
$$a=(a_1,\dots,a_{d_p})\in{\mathbf S}^{2d_p-1}\longmapsto S_a
=\sum_{j=1}^{d_p}a_jS_j^p\in{\mathcal S}^p,$$
and we let $\lambda_p$ be the probability measure on 
$\mathcal S^p$ induced by the normalized surface measure on 
${\mathbf S}^{2d_p-1}$, denoted also by $\lambda_p$ (i.e.,\ 
$\lambda_p({\mathbf S}^{2d_p-1})=1$).
Consider the probability space 
${\mathcal S}_\infty=\prod_{p=1}^\infty{\mathcal S}^p$ 
endowed with the probability measure 
$\lambda_\infty=\prod_{p=1}^\infty\lambda_p$\,. 
We denote by $[\sigma=0]$ the current of integration
(with multiplicities) over the zero set of a nontrivial section 
$\sigma\in H^0_{(2)}(X,L_p)$.
\begin{Theorem}\label{T:a1}
Let $(X,\omega)$, $(L_p,h_p)$, $p\geq1$, verify 
assumptions (A)-(B) and assume that 
$\sum_{p=1}^\infty \frac{1}{A_p^{2}}<\infty$. Then for
$\lambda_\infty$-a.e. sequence 
$\{\sigma_p\}_{p\geq1}\in{\mathcal S}_\infty$ 
we have in the weak sense of currents on $X$ that 
\begin{equation*}
\lim_{p\to\infty}\,\frac{1}{A_p}
\big([\sigma_p=0]-c_1(L_p,h_p)\big)=0\,.
\end{equation*}
Moreover, if $\lim_{p\to\infty}\,\frac{1}{A_p}\,c_1(L_p,h_p)=T$ 
for some positive closed current $T$ of bidegree $(1,1)$ on $X$, 
then for $\lambda_\infty$-a.e. sequence
$\{\sigma_p\}_{p\geq1}\in{\mathcal S}_\infty\,$,
\begin{equation*}
\lim_{p\to\infty}\,\frac{1}{A_p}[\sigma_p=0]=T\;
\text{ weakly on $X$}.
\end{equation*}
\end{Theorem}
\begin{proof} The arguments in \cite{ShZ99,ShZ08}
    (see also \cite[Section 5.2]{CM13}; in all these papers 
    $A_p=p$) imply that for $\lambda_\infty$-a.e. sequence
    $\{\sigma_p\}_{p\geq1}\in{\mathcal S}_\infty\,$,
$$\lim_{p\to\infty}\,\frac{1}{A_p}\,\big([\sigma_p=0]
-\gamma_p\big)=0$$
weakly in the sense of currents on $X$. Indeed, by working 
with a countable set of test forms and since 
$$\int_X[\sigma_p=0]\wedge\omega^{n-1}
=\int_X\gamma_p\wedge\omega^{n-1}
=\int_X c_1(L_p,h_p)\wedge\omega^{n-1}=A_p\,,$$
it suffices to show that, for a fixed test form $\theta$, one has
\begin{equation}\label{e:SZ1}
\lim_{p\to\infty}\,\frac{1}{A_p}\,\big\langle[\sigma_p=0]
-\gamma_p,\theta\,\big\rangle=0\,,
\end{equation}
for $\lambda_\infty$-a.e. $\sigma=\{\sigma_p\}_{p\geq1}
\in{\mathcal S}_\infty$. Let
$$Y_p:{\mathcal S}_\infty\longrightarrow\mathbb{C}\,,
\quad Y_p(\sigma)=\frac{1}{A_p}\,\big\langle[\sigma_p=0]
-\gamma_p,\theta\,\big\rangle\,.$$
The calculations in \cite[Sec. 3.1-3.3]{ShZ99} show that 
$$\int_{{\mathcal S}_\infty}Y_p\,d\lambda_\infty=0\,,\;\;
\int_{{\mathcal S}_\infty}|Y_p|^2\,d\lambda_\infty
\leq\frac{AC_\theta}{A_p^2}\,,\,\text{ with } A
=\frac{1}{\pi^2}\int_{{\mathbb C}^2}(\log|z_1|)^2
e^{-|z_1|^2-|z_2|^2}\,dz\,,$$
where $dz$ is the Lebesgue measure on ${\mathbb C}^2$, 
and $C_\theta$ is a constant depending only on $\theta$. 
Then \eqref{e:SZ1} follows since 
$$\int_{{\mathcal S}_\infty}\left(\sum_{p=1}^\infty|Y_p|^2\right)\,
d\lambda_\infty\leq A\,C_\theta\;\sum_{p=1}^\infty
\frac{1}{A_p^2}<+\infty\,.$$
We now conclude by Theorem \ref{T:mt1}.
\end{proof}

\par We show next that equidistribution results hold not only 
for the zeros of random sequences of holomorphic sections but
also for the logarithms of their pointwise norms. 

\begin{Theorem}\label{T:a2}
Let $(X,\omega)$, $(L_p,h_p)$, $p\geq1$, verify 
assumptions (A)-(B) and assume that
\begin{equation}\label{e:logdp}
\liminf_{p\to\infty}\,\frac{\log d_p}{A_p}\,=0\,.
\end{equation}
Then there exists an increasing sequence of natural numbers 
$\{p_j\}_{j\geq1}$ such that for $\lambda_\infty$-a.e. sequence 
$\{\sigma_p\}_{p\geq1}\in{\mathcal S}_\infty$ we have 
\begin{equation*}
\lim_{j\to\infty}\frac{\log|\sigma_{p_j}|_{h_{p_j}}}{A_{p_j}}=0\;\,
\text{ in $L^1(X,\omega^n)$}.
\end{equation*}
\end{Theorem}

\par To prove the theorem we need the following elementary lemma whose 
proof is included for the convenience of the reader.

\begin{Lemma}\label{L:Ik} If ${\mathbf S}^{2k-1}$ is the unit 
sphere in ${\mathbb C}^k$ with surface measure $d{\mathcal A}$ and 
$${\mathbf I}(k):=-\frac{1}{{\rm area}({\mathbf S}^{2k-1})}\,
\int_{{\mathbf S}^{2k-1}}\log|z_k|\,d{\mathcal A}\,,$$
where $z=(z_1,\ldots,z_k)\in{\mathbb C}^k$, then there exist 
numbers $a,b>1$ such that
$${\mathbf I}(k)\leq a\log k+b\;,\;\;\forall\,k\geq1\,.$$
\end{Lemma}

\begin{proof} We use spherical coordinates 
$(\theta_1,\ldots,\theta_{2k-2},\varphi)\in\left[-\frac{\pi}{2}, \frac{\pi}{2}\right]^{2k-2}\times[0,2\pi]$ 
on ${\mathbf S}^{2k-1}$ such that 
$$z_k=\sin\theta_{2k-3}\cos\theta_{2k-2}+i\sin\theta_{2k-2}\,,\;\;
d{\mathcal A}=\cos\theta_1\cos^2\theta_2\ldots\cos^{2k-2}\theta_{2k-2}\;d\theta_1\ldots d\theta_{2k-2}d\varphi\,.$$
We obtain 
$${\mathbf I}(k)=\frac{J(k)}{2C_{2k-2}C_{2k-3}}\,, \text{ where } C_k=\int_0^{\pi/2}\cos^kt\,dt\,,$$ 
and, after performing the change of variables $x=\sin\theta_{2k-2}\,$, $y=\sin\theta_{2k-3}$, 
$$J(k)=-\int_0^1\int_0^1(1-x^2)^{k-3/2}(1-y^2)^{k-2}\log(x^2+y^2-x^2y^2)\,dxdy\,.$$
Since $2x^2y^2\leq x^2+y^2$ for $0\leq x,y\leq1$, we have 
\begin{eqnarray*}
-\log(x^2+y^2-x^2y^2)&\leq&-\log\frac{x^2+y^2}{2}\;,\\
(1-x^2)^{k-3/2}(1-y^2)^{k-2}&\leq&(1-x^2-y^2+x^2y^2)^{k-2}\leq\left(1-\frac{x^2+y^2}{2}\right)^{k-2}\;,
\end{eqnarray*}
provided that $k\geq2$. If $D=\{(x,y)\in{\mathbb R}^2:\,x\geq0,\,y\geq0,\,x^2+y^2\leq2\}$ it follows that 
\begin{eqnarray*}
J(k)&\leq&-\iint_D\left(1-\frac{x^2+y^2}{2}\right)^{k-2}\log\frac{x^2+y^2}{2}\;dxdy\\
&=&-\frac{\pi}{2}\int_0^{\sqrt2}r\left(1-\frac{r^2}{2}\right)^{k-2}\log\frac{r^2}{2}\;dr
=-\frac{\pi}{2}\int_0^1(1-t)^{k-2}\log t\,dt\;.
\end{eqnarray*}
One shows by induction on $k\geq2$ that 
$$-\int_0^1(1-t)^{k-2}\log t\,dt=\frac{1}{k-1}\,\left(1+\frac{1}{2}+\ldots+\frac{1}{k-1}\right)\leq\frac{1+\log(k-1)}{k-1}\;.$$
The lemma follows since there exist constants $A,B>0$ such that $A\leq C_k\sqrt{k}\leq B$ for all $k\geq1$. Indeed, an integration by parts shows that 
$$C_{k+2}=\frac{k+1}{k+2}\,C_k\,,\,\text{ so } C_{2k}=\frac{\pi\,(2k)!}{2^{2k+1}(k!)^2}\,,\,\;
C_{2k+1}=\frac{2^{2k}(k!)^2}{(2k+1)!}\,,\,\text{ for } k\geq0\,.$$
By Stirling's formula $k!\sim (k/e)^k\sqrt{k}$, $(2k)!\sim(2k/e)^{2k}\sqrt{2k}\,$, which implies our claim.
\end{proof}

\smallskip

\begin{proof}[Proof of Theorem \ref{T:a2}]
Using \eqref{e:logdp} we can find a sequence of integers 
$p_j\nearrow\infty$ such that 
$\sum_{j=1}^\infty\frac{\log d_{p_j}}{A_{p_j}}<\infty$. 
We define 
$$Y_p:{\mathcal S}_\infty\longrightarrow{\mathbb R}\,,\;
Y_p(\sigma)=\frac{1}{A_p}\int_X\log\frac{|\sigma_p|_{h_p}}{
\sqrt{P_p}}\;\omega^n\;,\;\;\text{where}\;\sigma
=\{\sigma_p\}_{p\geq1}\,.$$
By Theorem \ref{T:mt1} we have $\frac{1}{A_p}\,\log P_p\to 0$
as $p\to\infty$, in $L^1(X,\omega^n)$. Since 
$\log\frac{|\sigma_p|_{h_p}}{\sqrt{P_p}}\leq0$ on $X$ for 
$\sigma=\{\sigma_p\}_{p\geq1}\in{\mathcal S}_\infty$, 
it suffices to show that $Y_{p_j}(\sigma)\to0$ as $j\to\infty$
for $\lambda_\infty$-a.e. $\sigma\in{\mathcal S}_\infty$.
By Tonelli's theorem we have
$$\int_{{\mathcal S}_\infty}Y_p(\sigma)\,d\lambda_\infty
=\frac{1}{A_p}\int_{{\mathcal S}^p}
\left(\int_X\log\frac{|\sigma_p|_{h_p}}{\sqrt{P_p}}\;
\omega^n\right)d\lambda_p=\frac{1}{A_p}
\int_X\left(\int_{{\mathcal S}^p}\log\frac{|\sigma_p|_{h_p}}{
\sqrt{P_p}}\;d\lambda_p\right)\omega^n\,.$$
For a fixed $x\in X$, we write $S_l^p=s_l^pe_p$ for
some holomorphic frame $e_p$ of $L_p$ near $x$ and we set 
$$u=(u_1,\ldots,u_{d_p})\;,\;\;u_l=\frac{s_l^p}{\sqrt{|s_1^p|^2
+\ldots+|s_{d_p}^p|^2}}\;.$$
Then the integral 
$$\int_{{\mathcal S}^p}\log\frac{|\sigma_p(x)|_{h_p}}{
\sqrt{P_p(x)}}\;d\lambda_p=\int_{{\mathbf S}^{2d_p-1}}
\log|a\cdot u(x)|\,d\lambda_p(a)=-{\mathbf I}(d_p)$$
is independent of $x$, where 
$a\cdot u=a_1u_1+\ldots+a_{d_p}u_{d_p}$ and ${\mathbf I}(k)$ 
is defined in Lemma \ref{L:Ik}. Using Lemma \ref{L:Ik} it follows that 
$$\int_{{\mathcal S}_\infty}Y_p(\sigma)\,d\lambda_\infty
\geq-\frac{a\log d_p+b}{A_p}\,\int_X\omega^n\,.$$
The definition of the sequence $\{p_j\}_{j\geq1}$ shows that 
$$\sum_{j=1}^\infty\int_{{\mathcal S}_\infty}Y_{p_j}(\sigma)\,
d\lambda_\infty>-\infty\,.$$ 
Since $Y_p\leq0$ this implies that $\sum_{j=1}^\infty Y_{p_j}$ 
converges in $L^1({\mathcal S}_\infty,\lambda_\infty)$, hence 
$Y_{p_j}(\sigma)\to0$ as $j\to\infty$ for $\lambda_\infty$-a.e. 
$\sigma\in{\mathcal S}_\infty$.
\end{proof}

\par Let us give two general examples in which condition 
\eqref{e:logdp} holds true.  

\begin{Proposition}\label{P:logdp1}
Let $(X,\omega)$, $(L_p,h_p)$, $p\geq1$, verify 
assumptions (A)-(B) and assume that $X$ is smooth and that 
each line bundle $L_p$ has a continuous metric $h'_p$ with 
the following property: every $x\in X$ has a contractible Stein 
coordinate neighborhood $W_x$ on which each metric $h'_p$ 
has a weight $\psi'_p$ such that the family 
$\{\psi'_p/A_p\}_{p\geq1}$ is equicontinuous on $W_x$. Then 
$$\lim_{p\to\infty}\frac{\log\dim H^0(X,L_p)}{A_p}=0\,.$$
\end{Proposition}

\begin{proof} Let $\varepsilon>0$ and let $P'_p$ be the 
    Bergman kernel function of the space $H^0(X,L_p)$ with 
    respect to the metrics $h'_p$ and $\omega$. For $x\in X$ 
    fix $r_x>0$ so that the (closed) ball $B(x,2r_x)\subset W_x$ 
    and let $U_x=B(x,r_x)$. The proof of the upper bound
    in \eqref{e:Bkf} works for any metric on $L_p$ 
    (see also \cite[(7)]{CM11}) and shows that 
$$P'_p(z)\leq C_xr^{-2n}\exp\left(2\max_{B(z,r)}\psi'_p
-2\psi'_p(z)\right),$$
for any $p\geq1$, $r<r_x$ and $z\in U_x$, where $C_x$ is
a constant depending only on $x$. The equicontinuity assumption 
implies that there exists $r_1=r_1(x,\varepsilon)<r_x$ such that
$2\max_{B(z,r_1)}\psi'_p-2\psi'_p(z)\leq A_p\varepsilon$ for all 
$p\geq1$ and $z\in U_x$, hence $P'_p(z)\leq C_xr_1^{-2n}
\exp(A_p\varepsilon)$. A standard compactness argument now 
shows that there exists a constant $C'=C'(\varepsilon)>0$ 
such that $P'_p\leq C'\exp(A_p\varepsilon)$ holds on $X$ for 
all $p\geq1$. It follows that 
$$\dim H^0(X,L_p)=\int_XP'_p\,\frac{\omega^n}{n!}\leq C'\exp(A_p
\varepsilon)\int_X\frac{\omega^n}{n!}\;,\;\;p\geq1\,,$$
which implies the conclusion of the proposition. 
\end{proof}

\par The second general example is provided by the class of 
semi-ample line bundles. Recall that a line bundle $L$ on $X$ is 
called semi-ample if $L^k$ is globally generated for some $k>0$, 
or, equivalently, the space $H^0(X,L^k)$ has no base locus. 

\begin{Proposition}\label{P:logdp2}
Let $(X,\omega)$, $(L_p,h_p)$, $p\geq1$, verify 
assumptions (A)-(B) and assume that $X$ is smooth and that each
line bundle $L_p$ is semi-ample. Then there exist an integer $N>0$ 
and a constant $C>0$ depending only on $\omega$ such that 
$$\dim H^0(X,L_p)\leq CA_p^N\;,\;\;\forall\,p\geq1\,.$$
\end{Proposition}

\begin{proof} Since $L_p$ is big, $X$ is Moishezon and hence 
    projective since it is K\"ahler. By the main theorem in
    \cite{KolMat}, there exists a polynomial $Q(y,z)$ depending 
    only on $\dim X$ such that for any semi-ample line bundle $L$ 
    on $X$ one has that 
$$\dim H^0(X,L)\leq Q\left(\int_Xc_1(L)^n\;,
\int_Xc_1(L)^{n-1}\wedge c_1(X)\right).$$
Lemma \ref{L:cohom} below implies that there 
exists a constant $C'>0$ depending only on $(X,\omega)$ such 
that 
$$\int_Xc_1(L_p)^n\leq C'A_p^n\;,\;\;\int_Xc_1(L_p)^{n-1}
\wedge c_1(X)\leq C'A_p^{n-1}\,.$$
The conclusion now follows.
\end{proof}

\begin{Lemma}\label{L:cohom} Let $(X,\omega)$ be a compact
    K\"ahler manifold of dimension $n$ and let $\beta$ be a real 
    valued closed form of type $(1,1)$ on $X$. Then there exists 
     $C>0$ depending only on $\omega$ and $\beta$ 
    such that for any pseudoeffective class 
    $\alpha\in H^{1,1}(X,{\mathbb R})$ we have 
$$\int_X\alpha^k\wedge\beta^{n-k}\leq C\|\alpha\|^k\;,\;k=1,
\dots,n\;,\;\;\text{ where } \|\alpha\|=\int_X\alpha\wedge
\omega^{n-1}.$$ 
\end{Lemma}

\begin{proof} If $\|\alpha\|=0$ then $\alpha=0$ since $\alpha$
    is pseudoeffective, so we can assume $\|\alpha\|>0$. 
    Let $\theta\in\alpha$ be a smooth form and 
    $T=\theta+dd^c\varphi$ be a positive closed current,
    where $\varphi$ is a $\theta$-psh function. 
    The Lelong numbers $\nu(T,x)<C_1\|\alpha\|$ for all $x\in X$, 
    where $C_1$ is a constant depending only on $\omega$
    (see e.g., \cite[Lemma 2.5]{Bo02}). Demailly's regularization 
    theorem \cite{D92} shows that there exists a sequence of 
    smooth functions $\varphi_k\searrow\varphi$ such 
    $\theta+dd^c\varphi_k\geq-C_2\lambda_k\omega$, where 
    $\lambda_k$ are continuous functions on $X$, 
    $\lambda_k(x)\searrow\nu(T,x)$ as $k\to\infty$ for every 
    $x\in X$, and $C_2$ is a constant depending only on $\omega$. 
    We fix $k$ so that $\lambda_k(x)<C_1\|\alpha\|$ for every 
    $x\in X$ and let $R=\theta+dd^c\varphi_k$ and 
    $R'=R+C_3\|\alpha\|\omega$, where $C_3=C_1C_2$, so 
    $R'\geq0$. Next we set $\beta'=\beta+c\, \omega$, where 
    $c>0$ is chosen so that $\beta'\geq0$. Since $R',\beta'\geq0$ 
    we obtain 
\begin{eqnarray*}
\int_X\alpha^k\wedge\beta^{n-k}&=&\int_XR^k\wedge\beta^{n-k}
=\int_X(R'-C_3\|\alpha\|\omega)^k\wedge(\beta'-c\omega)^{n-k}\\
&\leq&\int_X(R'+C_3\|\alpha\|\omega)^k
\wedge(\beta'+c\omega)^{n-k}
\leq C_4\Big\|R'+C_3\|\alpha\|\omega\Big\|^k\,,\end{eqnarray*}
where $C_4$ is a constant depending only on the K\"ahler form 
$\beta'+c\, \omega$ (hence on $\beta$ and $\omega$). 
The lemma follows since 
$$\Big\|R'+C_3\|\alpha\|\omega\Big\|
=\int_X(R+2C_3\|\alpha\|\omega)\wedge\omega^{n-1}
=\left(1+2C_3\int_X\omega^n\right)\|\alpha\|\,.$$
\end{proof}

\medskip

\par We conclude this section by discussing an application of
the above results to the problem of approximation of positive 
closed currents of bidegree $(1,1)$ on $X$ by currents of 
integration along analytic hypersurfaces of $X$. Let $\cA(X)$ 
be the space of positive closed currents 
$T\in\mathcal{D}'_{n-1,n-1}(X)$ with the property that there 
exist a sequence of singular Hermitian holomorphic line bundles 
$\{(F_p,h^{F_p})\}_{p\geq1}$ with 
$c_1(F_p,h^{F_p})\geq0$ and a sequence of natural numbers 
$N_p\to\infty$ such that 
\[\lim_{p\to\infty}\,\frac{1}{N_p}\,c_1(F_p,h^{F_p})=T\,.\]
If $X_{\rm sing}\neq\emptyset$ we require in addition that there 
exists a current $T_0\in\cT(X)$ (depending on $T$) such that
for all $p\geq1$ we have $\frac{1}{N_p}\,c_1(F_p,h^{F_p})
\leq T_0$.

\par When $X$ is smooth the space $\cA(X)$ is the closure in 
$\mathcal{D}'_{n-1,n-1}(X)$ of the convex cone generated by 
positive closed integral currents. Recall that a real closed current
$T\in\mathcal{D}'_{n-1,n-1}(X)$ is called integral if its 
de Rham cohomology class $[T]$ belongs to 
$H^{1,1}(X,\R)\cap H^2(X,\Z)$. A current $T$ is integral if 
and only if there exists a singular Hermitian holomorphic line bundle
$(L,h)$ on $X$ with $c_1(L,h)=T$ 
(see e.g., \cite[Lemma 2.3.5]{MM07}).
\begin{Theorem}\label{T:approx}
Let $(X,\omega)$ be a compact normal K\"ahler space and $(L,h)$ 
be a singular Hermitian holomorphic line bundle on $X$ such that 
$c_1(L,h)\geq\varepsilon\omega$ for some $\varepsilon>0$.
If $T\in\cA(X)$ then there exist a sequence of singular Hermitian 
holomorphic line bundles $\{(L_p,h_p)\}_{p\geq1}$ with
$c_1(L_p,h_p)\geq0$ and a sequence of natural numbers 
$N_p\to\infty$ such that for $\lambda_\infty$-a.e. sequence 
$\{\sigma_p\}_{p\geq1}\in{\mathcal S}_\infty\,$,
$$\lim_{p\to\infty}\,\frac{1}{N_p}[\sigma_p=0]=T\;
\text{ weakly on $X$}.$$
Here the probability space $({\mathcal S}_\infty,\lambda_\infty)$
is associated as above to the sequence $\{(L_p,h_p)\}_{p\geq1}$. 
\end{Theorem}
\begin{proof}
Since $T\in\cA(X)$ there exist line bundles 
$\{(F_p,h^{F_p})\}_{p\geq1}$ with  $c_1(F_p,h^{F_p})\geq0$ 
and a sequence of natural numbers $N_p\to\infty$, $p\to\infty$, 
such that $\lim_{p\to\infty}\,\frac{1}{N_p}\,c_1(F_p,h^{F_p})=T$. 
Moreover, if $X_{\rm sing}\neq\emptyset$ there exists a current
$T_0\in\cT(X)$ such that $c_1(F_p,h^{F_p})\leq N_pT_0$ for all
$p\geq1$. We can assume without loss of generality that 
$N_p\geq p$; otherwise replace $(F_p,h^{F_p})$ by 
$(F_p^{m_p},(h^{F_p})^{m_p})$ and $N_p$ by $m_pN_p$, 
with a convenient $m_p\in\N$. We fix a sequence $b_p\in\N$ 
such that $b_p\leq N_p$ for all $p\geq1$ and $b_p\to\infty$, 
$b_p/N_p\to0$ as $p\to\infty$. Let
\[L_p=F_p\otimes L^{b_p}\,,\quad 
h_p=h^{F_p}\otimes h^{b_p}\,.\]
The conclusion follows from Theorem \ref{T:a1} since 
$c_1(L_p,h_p)\geq b_p\,c_1(L,h)\geq b_p\,\varepsilon\omega$ 
and 
\begin{equation*}
\begin{split}
&c_1(L_p,h_p)\leq N_pT_0+b_p\,c_1(L,h)
\leq N_p\big(T_0+c_1(L,h)\big)\,,\\
&\frac{c_1(L_p,h_p)}{N_p}\to T\,,\quad 
\frac{\|c_1(L_p,h_p)\|}{N_p}\to \|T\|\,,\:\:\text{as $p\to\infty$}\,.
\end{split}
\end{equation*} 
\end{proof}

\par For related approximation results on compact K\"ahler manifolds we refer to \cite{D93b,D82i,G99,CM13b}. We note that Dinh and Sibony \cite{DS06} developed a general method to obtain the asymptotic distribution with speed of convergence of zeros of random sequences of holomorphic sections.

\section{Applications}\label{S:appl}

\subsection{Powers of a line bundle}\label{SS:appl1}
Theorems \ref{T:mt1}, \ref{T:a1}, \ref{T:a2} apply for the sequence $(L_p,h_p)=(L^p,h^p)$, where $(L,h)$ is a singular Hermitian holomorphic line bundle on $X$ with strictly positive curvature current.

\begin{Corollary}\label{C:Lp1}
Let $(X,\omega)$ be a compact normal K\"ahler space and $(L,h)$ 
be a singular Hermitian holomorphic line bundle on $X$ such that 
$c_1(L,h)\geq\varepsilon\omega$ for some $\varepsilon>0$. 
Then as $p\to\infty$:

\smallskip

\par (i) $\frac{1}{p}\,\log P_p\to 0$ in $L^1(X,\omega^n)$. 

\smallskip

\par (ii) $\frac{1}{p}\gamma_p\to c_1(L,h)$ weakly on $X$.
\smallskip

\par (iii) 
$\frac{1}{p}\big[\sigma_p=0\big]\to c_1(L,h)$ weakly on $X$,
for $\lambda_\infty$-a.e.\ sequence 
$\{\sigma_p\}_{p\geq1}\in{\mathcal S}_\infty$.
\end{Corollary}
Indeed, assumptions (A) and (B) are satisfied with 
$a_p=p\,\varepsilon$ and $T_0=c_1(L,h)/\|c_1(L,h)\|$, 
where $\|c_1(L,h)\|:=\int_Xc_1(L,h)\wedge\omega^{n-1}$. Moreover $A_p=p\,\|c_1(L,h)\|$.

\medskip

\par We consider now the case when the curvature current 
of the singular metric is not necessarily K\"ahler. 

\begin{Corollary}\label{C:Lp2}
Let $(L,h)$ be a singular Hermitian holomorphic line bundle on 
the compact 
normal K\"ahler space $(X,\omega)$ such that $c_1(L,h)\geq0$ 
and assume that $L$ 
has a singular metric $h_0$ with 
$c_1(L,h_0)\geq\varepsilon\omega$ for some $\varepsilon>0$. 
Let $\{n_p\}_{p\geq1}$ be a sequence of natural numbers such 
that 
\begin{equation}\label{n_p}
n_p\to\infty\:\:\text{and $n_p/p\to0$ as $p\to\infty$}.
\end{equation} 
Let 
\begin{equation}\label{h_p}
h_p=h^{p-n_p}\otimes h_0^{n_p}
\end{equation}
 and $P_p\,$, 
$\gamma_p\,$, ${\mathcal S}^p$ be the Bergman kernel function, 
Fubini-Study current, 
and respectively the unit sphere, associated to the spaces 
$H^0_{(2)}(X,L_p)=H^0_{(2)}(X,L^p,h_p)$. Then, as $p\to\infty$, 
we have 
$\frac1p\,\log P_p\to 0$ in $L^1(X,\omega^n)$, 
$\frac1p\,\gamma_p\to c_1(L,h)$ and 
$\frac{1}{p}\big[\sigma_p=0\big]\to c_1(L,h)$ in the sense of 
currents on $X$, 
for $\lambda_\infty$-a.e.\ sequence 
$\{\sigma_p\}_{p\geq1}\in{\mathcal S}_\infty$.
\end{Corollary}

\begin{proof} Note that 
$$c_1(L^p,h_p)=(p-n_p)c_1(L,h)+n_pc_1(L,h_0)\geq
\varepsilon n_p\omega\,,\;\frac{A_p}{p}=\|c_1(L,h)\|
=\|c_1(L,h_0)\|>0\,.$$ 
Hence 
$$c_1(L^p,h_p)\leq p(c_1(L,h)+c_1(L,h_0))=A_pT_0\,,\;
\text{ with }\,T_0=
\frac{1}{\|c_1(L,h)\|}\,(c_1(L,h)+c_1(L,h_0))\,,$$
and Theorems \ref{T:mt1} and \ref{T:a1} apply in this setting 
and conclude the proof.
\end{proof}

Corollary \ref{C:Lp1} applies to the case of the Satake-Baily-Borel 
compactifications of arithmetic quotients. 
Let $D$ be a bounded symmetric domain in $\mathbb{C}^n$ and
let $\Gamma$ be a 
torsion-free arithmetic lattice. 
Then $U:=D/\Gamma$ is a smooth quasi-projective variety, called
an arithmetic variety.
The Satake-Borel-Baily compactification of $U$ is a normal compact 
analytic space $X$ 
containing $U$ as a Zariski open dense set
and which is minimal with this property,
that is, given any normal compactification $U\subset X'$, the 
identity map on $U$
extends to a holomorphic map $X'\to X$ (see Satake \cite{Sa60},
Borel-Baily \cite{BB66}).
Moreover, $\codim(X\setminus U)\geq 2$. 

The Bergman metric $\beta_D$ 
on $D$ descends to a complete K\"ahler metric $\beta_U$ 
on $U$, which is K\"ahler-Einstein with 
$\ric_{\beta_U}=-\beta_U$. 
We denote by $h^{K_U}$ the Hermitian metric induced by $\beta$
on $K_U$.
Then $c_1(K_U,h^{K_U})=-\ric_{\beta_U}=\beta_U$.

\begin{Lemma}\label{l:be}
Assume that $\Gamma$ is neat.
There exists an ample line bundle $F\to X$ such that $F|_U=K_U$
and a singular metric $h^F$ on $F$ such that $h^F|_U=h^{K_U}$, 
such that
$c_1(F,h^F)$ is a K\"ahler current on $X$ and 
$c_1(F,h^F)|_U=\beta_U$.
Hence the Bergman K\"ahler metric $\beta_U$ extends to
a K\"ahler current $\beta_X$
on $X$.
\end{Lemma}
\begin{proof}
The existence of an ample line bundle $F\to X$ such that
$F|_U=K_U$ is shown in 
\cite[Proposition\,3.4(b)]{Mum77}; the bundle $F$ is the bundle
used in \cite[Theorem\,10.11]{BB66} to embed $X$ into 
a projective space. Since $c_1(K_U,h^{K_U})=\beta_U$, 
the local weights of $h^{K_U}$ are
psh, so they extend to $X$. Thus, the metric $h^{K_U}$ extends 
to a metric $h^F$ on $F\to X$ and $c_1(K_U,h^{K_U})$ extends 
to a positive closed current $c_1(F,h^{F})$ on $X$. 

To see that $c_1(F,h^{F})$ is a K\"ahler current, one can describe 
$F$ in the following way. By \cite{AMRT:10}, $U$ admits a smooth 
toroidal compactification $\wi{X}$. In particular, 
$\Sigma:=\wi{X}\setminus U$ 
is a divisor with normal crossings. Let 
$L:=K_{\wi{X}}\otimes\cO_{\wi{X}}(\Sigma)$.
Then the metric $h^{K_U}$ defines a singular metric $h^L$ on $L$
such that $c_1(L,h^L)$ is a 
closed positive current on $\wi{X}$ which extends 
$\beta_U/(2\pi)$, cf.\ \cite[Lemma\,6.8]{CM11}.
It follows from \cite[Proposition\,3.4]{Mum77} 
(see also \cite[Section 6.4]{CM11}) 
that $h^{K_U}$ is a good metric in the sense of Mumford, 
so that $\beta_U$ has Poincar\'e growth near $\Sigma$. 
Therefore, $c_1(L,h^L)$ is a K\"ahler current on $\wi{X}$.
If $\tau:\wi{X}\to X$ is the unique holomorphic map extending 
the identity on $U$, we have $L=\tau^*F$. Hence 
$c_1(F,h^F)=\tau_*c_1(L,h^L)$, so $c_1(F,h^F)$
is a K\"ahler current on $X$.
\end{proof}
\begin{Corollary}\label{C:Lp1.1}
Let $X$ be the Satake-Baily-Borel compactification of a smooth
arithmetic variety $U=D/\Gamma$, where $\Gamma$ is neat.
Let $(F,h^F)$ be the extension of $(K^U,h^{K_U})$ given by 
Lemma \ref{l:be}, where $h^{K_U}$ is induced by the Bergman 
metric $\beta_U$.
Then the conclusions of Corollary \ref{C:Lp1} hold for 
$c_1(F,h^F)=\beta_X$.
\end{Corollary}
Note that $H^0_{(2)}(X,F^p)=H^0_{(2)}(U,K_U^p)=
H^0(\wi{X},K_{\wi{X}}^p\otimes\cO_{\wi{X}}(\Sigma)^{p-1})=
H^0(\wi{X},L^p\otimes\cO_{\wi{X}}(\Sigma)^{-1})$
is the space of cusp forms of weight $p$ , 
cf.\ \cite[Lemma\,6.11]{CM11}. 
Hence Corollary \ref{C:Lp1.1} says that the Bergman kernel 
$P_p$ of the space of cusps forms satisfies 
$\frac{1}{p}\,\log P_p\to 0$ 
in $L^1(X,\omega^n)$, where $\omega$ is any smooth 
K\"ahler form on $X$. Moreover, the normalized zero-currents of 
random cusp forms distribute on $X$ to the extension $\beta_X$ 
of the Bergman metric $\beta_U$ on $U$. 

\smallskip

We consider next a more general situation as above. 
Let $X$ be a projective variety of 
general type with only canonical
singularities such that the canonical divisor $\mathcal{K}_X$ is 
an ample $\mathbb{Q}$-divisor. 
There exists an integer $\ell>0$ such that $\mathcal{K}_X^\ell$ 
is integral and the holomorphic line bundle 
$L=\cO_X(\mathcal{K}_X^\ell)\to X$ admits a smooth Hermitian 
metric $h^L$ such that
$c_1(L,h^L)=\omega$ is a smooth K\"ahler form on $X$.
By \cite[Theorem\,7.8]{EGZ09} there exists a unique 
$\varphi\in PSH(X,\omega)\cap L^\infty(X)$
such that 
\begin{equation}\label{KEs}
c_1(L,e^{-2\varphi}h^L)=\omega+dd^c\varphi
=\omega_\varphi
\end{equation}
 is a singular 
K\"ahler-Einstein metric in the sense of \cite[Definition\,7.3]{EGZ09}. 
In particular, $\omega_\varphi$ is a closed positive current 
on $X$ and its restriction to $X_{\rm reg}$ is a smooth 
K\"ahler-Einstein metric of negative curvature.
Corollary \ref{C:Lp2} yields the following.
\begin{Corollary}\label{C:Lp1.2}
Let $X$ be a projective variety of general type with only canonical 
singularities and let $\omega_\varphi$ be the singular 
K\"ahler-Einstein metric \eqref{KEs}.
Let $h_p$ be the metric on $L^p$ 
constructed as in \eqref{h_p} using a sequence $\{n_p\}_{p\geq1}$ as in \eqref{n_p}
and the metrics $h_0=h^L$ 
and $h=e^{-2\varphi}h^L$ on $L$. 
Let $\gamma_p\,$, ${\mathcal S}^p$ be the Fubini-Study current 
and the unit sphere associated 
to the space $H^0_{(2)}(X,L^p,h_p)$. Then 
$\frac{1}{p}\gamma_p\to\omega_\varphi$\,,
$\frac{1}{p}\big[\sigma_p=0\big]\to\omega_\varphi$
as $p\to\infty$ weakly on $X$, for $\lambda_\infty$-a.e.\ sequence 
$\{\sigma_p\}_{p\geq1}\in{\mathcal S}_\infty$. 
\end{Corollary}
A similar discussion applies to canonically polarized klt compact 
K\"ahler pairs $(X,\Delta)$ considered in 
\cite[Theorem\,7.12]{EGZ09} or to $\mathbb{Q}$-Fano K\"ahler 
spaces (that is, klt compact K\"ahler spaces with $-\mathcal{K}_X$ 
ample $\mathbb{Q}$-divisor).

\subsection{Powers of ample line bundles}\label{SS:appl2} 
We specialize in the sequel the results of the previous corollary
to the case when $(X,\omega)$ is a projective manifold with 
a polarization $(L,h_0)$, where $L$ is a positive line bundle on $X$
endowed with a smooth Hermitian metric $h_0$ such that 
$c_1(L,h_0)=\omega$. The set of singular Hermitian metrics 
$h$ on $L$ with $c_1(L,h)\geq0$ is in one-to-one correspondence 
to the set $PSH(X,\omega)$ of $\omega$-plurisubharmonic 
($\omega$-psh) functions on $X$, by associating to 
$\varphi\in PSH(X,\omega)$ the metric 
$h_\varphi=e^{-2\varphi}h_0$ (see e.g., \cite{D90,GZ05}). 
Note that 
$$c_1(L,h_\varphi)=\omega+dd^c\varphi=:\omega_\varphi\,.$$

\begin{Corollary}\label{C:Lp3}
Let $(X,\omega)$ be a compact K\"ahler manifold and $(L,h_0)$ 
be a positive line bundle on $X$ 
with $c_1(L,h_0)=\omega$. Let $\varphi\in PSH(X,\omega)$ and 
$h_p$ be a metric on $L^p$ 
constructed as in \eqref{h_p} using a sequence $\{n_p\}_{p\geq1}$ as in \eqref{n_p}
and the metrics $h_0$ and $h=h_\varphi$ on $L$. 
Let $\gamma_p\,$, ${\mathcal S}^p$ be the Fubini-Study 
current and the unit sphere associated 
to the space $H^0_{(2)}(X,L^p,h_p)$. Then 
$\frac{1}{p}\gamma_p\to\omega_\varphi$\,,
$\frac{1}{p}\big[\sigma_p=0\big]\to\omega_\varphi$ 
as $p\to\infty$ weakly on $X$, 
for $\lambda_\infty$-a.e.\ sequence 
$\{\sigma_p\}_{p\geq1}\in{\mathcal S}_\infty$. Moreover,
if $\varphi$ is continuous then 
$$\frac{1}{p^k}\,\gamma_p^k\to\omega_\varphi^k=
(\omega+dd^c\varphi)^k\;\text{ weakly on }X,\,
\text{ for }k\leq n\,.$$
\end{Corollary}

\begin{proof} The first conclusion follows directly from 
    Corollary \ref{C:Lp2}. If $\varphi$ 
is continuous and $P_p$ is the Bergman kernel function of 
the space $H^0_{(2)}(X,L^p,h_p)$ 
one can proceed as in the proof of \cite[Theorem 5.3]{CM11} 
to show that $\frac1p\,\log P_p\to0$ 
uniformly on $X$. This implies the second conclusion of 
the corollary, as in \cite[Theorem 5.4]{CM11}. 
\end{proof}

\par Note that since $h_0$ is smooth we have that 
$H^0_{(2)}(X,L^p,h^p)\subset H^0_{(2)}(X,L^p,h_p)$. 
Moreover, if the metric $h=h_\varphi$ is bounded
(i.e.,\ $\varphi$ is bounded) then equality holds, 
$H^0_{(2)}(X,L^p,h^p)=H^0_{(2)}(X,L^p,h_p)=H^0(X,L^p)$. 

\par We remark now that, instead of working with random 
sections of spheres, one can identify $H^0_{(2)}(X,L_p)$ to 
${\mathbb C}^{d_p}$ by 
$$a=(a_1,\dots,a_{d_p})\in{\mathbb C}^{d_p}\longmapsto S_a
=\sum_{j=1}^{d_p}a_jS_j^p\in H^0_{(2)}(X,L_p)\,,$$
and one can consider $a_j$, $1\leq j\leq d_p$, as independent 
identically distributed Gaussian random variables on $\mathbb C$.
Thus the probability space $({\mathcal S}^p,\lambda_p)$
is replaced by $(H^0_{(2)}(X,L_p),\mu_p)$, where
$$d\mu_p(z)=\pi^{-d_p}e^{-(|z_1|^2
+\ldots+|z_{d_p}|^2)}dm(z)\,,$$
and $dm(z)$ is the Lebesgue measure on ${\mathbb C}^{d_p}$. 
Let $\mu_\infty=\prod_{p=1}^\infty\mu_p$ be the product 
measure on the space 
${\mathcal H}=\prod_{p=1}^\infty H^0_{(2)}(X,L_p)$. 
Since the measure $\mu_p$ is unitary invariant, Theorems \ref{T:a1} and \ref{T:a2} hold in this context with similar proofs. More precisely, we have the following:

\begin{Theorem}\label{T:a3}
Let $(X,\omega)$, $(L_p,h_p)$, $p\geq1$, verify assumptions (A)-(B). 

(i) If $\sum_{p=1}^\infty \frac{1}{A_p^{2}}<\infty$ then for $\mu_\infty$-a.e. sequence $\{\sigma_p\}_{p\geq1}\in{\mathcal H}$ we have that 
\begin{equation*}
\lim_{p\to\infty}\,\frac{1}{A_p}
\big([\sigma_p=0]-c_1(L_p,h_p)\big)=0,\text{ in the weak sense of currents on } X.
\end{equation*}

(ii) If condition \eqref{e:logdp} holds then there exists an increasing sequence of natural numbers 
$\{p_j\}_{j\geq1}$ such that for $\mu_\infty$-a.e. sequence 
$\{\sigma_p\}_{p\geq1}\in{\mathcal H}$ we have 
\begin{equation*}
\lim_{j\to\infty}\frac{\log|\sigma_{p_j}|_{h_{p_j}}}{A_{p_j}}=0\;\,
\text{ in $L^1(X,\omega^n)$}.
\end{equation*}
\end{Theorem}

Hence Corollary \ref{C:Lp3} can be seen as a generalization of 
Theorem 5.2 in \cite{BL13} which deals with the special case 
when $\varphi={\mathcal V}_{K,q}^*$ is the weighted 
$\omega$-psh global extremal function of a compact $K\subset X$.
Using different methods and a different inner product on 
$H^0(X,L^p)$ (note that $\mathcal{V}_{K,q}^*$ is bounded) it
is shown in \cite[Theorem\,5.2]{BL13} that 
\[\frac1p\big[\sigma_p=0\big]\to
\omega+dd^c\mathcal{V}_{K,q}^*\:\:\text{for $\mu_\infty$-a.\,e. 
sequences $\{\sigma_p\}_{p\geq1}\in
\prod_{p=1}^{\infty}H^0(X,L^p)$}.\]
On the other hand \cite[Theorem\,5.2]{BL13}  holds for more 
general probability measures than 
$\mu_\infty$ (see \cite[(2.1) and (2.2)]{BL13}).

\begin{Remark} 
A particularly interesting case is when $X=\mathbb{P}^n$, 
$L=\mathcal{O}(1)$ and $\omega=\omega_{_\mathrm{FS}}$ 
is the Fubini-Study metric in Corollary \ref{C:Lp3}. In this case
the class $PSH(\mathbb{P}^n,\omega_{_\mathrm{FS}})$ is in 
one-to-one correspondence with the Lelong class of psh functions 
on $\mathbb{C}^n$ of logarithmic growth, and the sections in
$H^0(\mathbb{P}^n,\mathcal{O}(p))$ can be identified to 
polynomials of degree $\leq p$ on $\mathbb{C}^n$
(see e.g., \cite[Section 5]{BL13}). Therefore 
Corollary \ref{C:Lp3} yields a general equidistribution result for
the zeros of a random sequence of polynomials on 
$\mathbb{C}^n$. For related results see \cite{Ba13}, \cite{BL13} 
and references therein.
\end{Remark}

\smallskip

We consider now approximation of metrics with conic singularities
\cite{Do12,Ti96}.
Let $X$ be a Fano manifold, that is, the anti-canonical line bundle
$K_X^{-1}$ is ample. Fix a Hermitian metric $h_0$ on $K_X^{-1}$, 
such that $\omega:=c_1(K_X^{-1},h_0)$ is a K\"ahler metric.
Recall that Hermitian metrics on
$K_X^{-1}$ can be identified to volume forms on $X$.
Let $D$ be a smooth divisor in the linear system defined by 
$K_X^{-\ell}$, so there
exists a section $s\in H^0(X,K_X^{-\ell})$ with
$D=\operatorname{Div}(s)$.
We fix a smooth metric $h$ on the bundle $\cO_X(D)$. 
Let $\beta\in[0,1)$.

A K\"ahler metric $\widehat\omega$ on $X$ with cone angle 
$2\pi(1-\beta)$ along $D$ is a current $\widehat\omega\in c_1(X)$ 
such that 
$\widehat\omega=\omega_\varphi=\omega+dd^c\varphi$
where $\varphi=\psi+|s|_h^{2-2\beta}\in PSH(X,\omega)$ and
$\psi\in\cC^\infty(X)\cap PSH(X,\omega)$. 
In a neigbourhood of a point of $D$
where $D$ is given by $z_1=0$ the metric $\widehat\omega$
is equivalent
to the cone metric $\frac{i}{2}(|z_1|^{-2\beta}dz_1\wedge 
d\overline{z}_1+
\sum_{j=2}^ndz_j\wedge d\overline{z}_j)$.
The Riemannian metric determined by $\widehat\omega$ 
has conic singularity along
$D$ of conic angle $2\pi(1-\beta)$. 

The metric $\widehat\omega$ defines a 
singular metric $h_{\widehat\omega}$ on $K_X^{-1}$ with 
curvature current
$\ric_{\,\widehat\omega}:=c_1(K_X^{-1},h_{\widehat\omega})$.
The metric $\widehat\omega$ is called K\"ahler-Einstein with conic singularities
of cone angle $2\pi(1-\beta)$ along $D$ if $\ric_{\,\widehat\omega}=(1-\ell\beta)
\widehat\omega+\beta[D]$, where
$[D]$ is the current of integration on $D$.

\begin{Corollary}\label{C:Lp4}
Let $X$ be a Fano manifold and let $\widehat\omega$ be
a K\"ahler metric
with cone angle $\beta$.
Let $h_p$ be the metric on $K_X^{-p}$ 
constructed as in Corollary \ref{C:Lp2} using the metric 
$h=h_{\widehat\omega}$ on $K_X^{-1}$.
Let  $\gamma_p\,$, ${\mathcal S}^p$ be the Fubini-Study current 
and the unit sphere associated 
to the space $H^0_{(2)}(X,K_X^{-p},h_p)$. Then 
$\frac1p\gamma_p\to\ric_{\,\widehat\omega}$ and 
$\frac{1}{p}\big[\sigma_p=0\big]\to\ric_{\,\widehat\omega}$ 
as $p\to\infty$ weakly on $X$, for $\lambda_\infty$-a.e.\ sequence 
$\{\sigma_p\}_{p\geq1}\in{\mathcal S}_\infty$. 
If $\widehat\omega$ is a K\"ahler current, then we can take 
$h_p=h_{\widehat\omega}^p$ above. 
\end{Corollary}

\begin{Remark}
If $\widehat\omega$ is a K\"ahler current, 
\cite[Theorem\,1.8]{HsM11} shows
that the Bergman kernel of $H^0_{(2)}(X,K_X^{-p},
h_{\widehat\omega}^p)$ has a full 
asymptotic expansion in powers $p^{n-j}$, $j=0,1,\ldots$\,, and 
the Fubini-Study forms converge to $\ric_{\,\widehat\omega}$ 
in the $\cC^\infty$-topology on compact sets of $X\setminus D$. 
\end{Remark}


\subsection{Tensor products of powers of several line bundles}\label{SS:appl3} 
Another typical application of Theorem \ref{T:mt1} is the following.

\begin{Corollary}\label{C:prod}
Let $(X,\omega)$ be a compact normal K\"ahler space. 
Assume that $(F_j,h^{F_j})$, $1\leq j\leq k$, are singular 
Hermitian holomorphic line bundles with $c_1(F_j,h^{F_j})\geq0$ 
and $c_1(F_1,h^{F_1})\geq\varepsilon\omega$, for some 
$\varepsilon>0$. Let $T=\sum_{j=1}^k r_j\,c_1(F_j,h^{F_j})$, 
where $r_j\geq0$, and let $\{m_{j,p}\}_p\,$, $1\leq j\leq k$,
be sequences of natural numbers such that
\[
m_{1,p}\to\infty\,,\;\frac{m_{j,p}}{p}\to r_j\,,\;1\leq j\leq k\,,\;
\text{ as } p\to\infty\,.
\]
Let $P_p\,$, $\gamma_p\,$, ${\mathcal S}^p$ be the 
Bergman kernel function, Fubini-Study current, and respectively 
the unit sphere, associated to $H^0_{(2)}(X,L_p)$, where 
\[
L_p=F^{m_{1,p}}_1\otimes\ldots\otimes F^{m_{k,p}}_k\,,\;\,
h_p=(h^{F_1})^{m_{1,p}}\otimes\ldots
\otimes(h^{F_k})^{m_{k,p}}\,.
\]  
Then, as $p\to\infty$, we have $\frac1p\,\log P_p\to 0$ in
$L^1(X,\omega^n)$, $\frac1p\,\gamma_p\to T$ and 
$\frac{1}{p}\big[\sigma_p=0\big]\to T$ in the weak sense of
currents on $X$, for $\lambda_\infty$-a.e.\ sequence 
$\{\sigma_p\}_{p\geq1}\in{\mathcal S}_\infty$.
\end{Corollary}

\begin{proof} 
We may assume that $m_{j,p}/p<r_j+1$ for all $1\leq j\leq k$, 
$p\geq1$, so 
\begin{eqnarray*}
c_1(L_p,h_p)&=&\sum_{j=1}^k m_{j,p}\,c_1(F_j,h^{F_j})
\geq\varepsilon m_{1,p}\,\omega\,,\\
c_1(L_p,h_p)&\leq& p\,T_0\,,\;\text{ where }
T_0=\sum_{j=1}^k (r_j+1)\,c_1(F_j,h^{F_j})\,.
\end{eqnarray*}
Moreover $c_1(L_p,h_p)/p\to T$ as $p\to\infty$. Note that 
\[\frac{A_p}{p}=\frac1p\;\sum_{j=1}^k m_{j,p}\,
\|c_1(F_j,h^{F_j})\|\to\sum_{j=1}^k r_j\|c_1(F_j,h^{F_j})\|\,.\]
The conclusion follows from Theorems \ref{T:mt1} and \ref{T:a1}, 
since 
$$\frac{\gamma_p}{p}=\frac{A_p}{p}
\frac{1}{A_p}(\,\gamma_p-c_1(L_p,h_p))+\frac{1}{p}\,
c_1(L_p,h_p)\,.$$
The fact that $\frac{1}{p}\big[\sigma_p=0\big]\to T$ weakly on $X$, for $\lambda_\infty$-a.e.\ sequence $\{\sigma_p\}_{p\geq1}\in{\mathcal S}_\infty$, follows as in the proof of Theorem \ref{T:a1}, where we normalize all currents dividing by $p$, by using that $A_p\lesssim p$ and $\sum_{p=1}^\infty p^{-2}$ is finite.
\end{proof}

\medskip

\par Let us remark that Theorem \ref{T:a2} holds in the context of 
Corollaries \ref{C:Lp1}, \ref{C:Lp2}, \ref{C:Lp3}, \ref{C:prod}, 
since the condition \eqref{e:logdp} holds in all of these situations. 

\subsection{Covering manifolds}
Let $(X,\omega)$ be a compact K\"ahler manifold of dimension $n$ 
and $(L_p,h_p)$, $p\geq1$, be a sequence of singular Hermitian 
holomorphic line bundles on $X$ satisfying condition \eqref{e:pc}. Recall that $A_p=\int_Xc_1(L_p,h_p)\wedge\omega^{n-1}$. Let $q:\widetilde X\to X$ be a (paracompact) Galois covering of 
$X$, where $q$
is the canonical projection. Let us denote by 
$\widetilde\omega=q^*\omega$, 
$\widetilde L_p=q^*(L_p)$ and by $\widetilde h_p$
the metric on $\widetilde L_p$ which is
the pull-back of the metric $h_p$.   We let 
$H^0_{(2)}(\widetilde X,\widetilde L_p)$ 
be the Bergman space of $L^2$-holomorphic sections of 
$\widetilde L_p$ relative to the metric 
$\widetilde h_p$ and the volume form $\widetilde\omega^n/n!$ on 
$\widetilde X$, defined as
in \eqref{e:bs}, endowed with the obvious inner product.
We define the Bergman kernel function $\widetilde P_p$ and 
Fubini-Study
currents $\widetilde\gamma_p$ associated to 
$H^0_{(2)}(\widetilde X,\widetilde L_p)$ 
as in \eqref{e:BFS1}. In this context, $d_p\in\N\cup\{\infty\}$, 
and these objects are well defined even for $d_p=\infty$, 
see \cite[Lemmas\,3.1-3.2]{CM11}.

Note that $\widetilde\omega$ is a complete K\"ahler metric on $\widetilde X$ and 
$c_1(\widetilde L_p,\widetilde h_p)=q^*c_1(L_p,h_p)\geq a_p\widetilde\omega$. Since $\ric_\omega\geq-B\omega$ holds on $X$, for some constant $B>0$, and $q$ is a local biholomorphism, we have $\ric_{\widetilde\omega}\geq-B\widetilde\omega$ on $\widetilde X$. Moreover, if $K\Subset\widetilde X$ there exists a constant $C_K>0$ such that
$$\int_Kc_1(\widetilde L_p,\widetilde h_p)\wedge\widetilde\omega^{n-1}\leq C_KA_p\,.$$

A straightforward adaptation of the proofs given above yields:
\begin{equation}
\begin{split}
&\frac{1}{A_p}\,\log \widetilde P_p\to 0
\:\: \text{in $L^1_{loc}(\widetilde X,\widetilde\omega^n)$}\,,\\
&\frac{1}{A_p}\,(\widetilde\gamma_p
-c_1(\widetilde L_p,\widetilde h_p))\to 0
\:\: \text{weakly on $\widetilde X$}\,.
\end{split}
\end{equation}
Assume moreover that the metrics $h_p$ are of class $\cC^3$ 
and condition \eqref{e:3d} is fulfilled, that is, 
$\varepsilon_p:=\|h_p\|_3^{1/3}a_p^{-1/2}\to0$ as
$p\to\infty$\,. Then there exists $p_0\in\mathbb N$ such that 
for all $p>p_0$ 
\begin{equation}\label{e:Bexp01}
\left|\widetilde P_p\,\frac{\widetilde\omega^n}{
c_1(\widetilde L_p,\widetilde h_p)^n}-1\right|
\leq C\varepsilon_p^{2/3}\:\:\text{ on $\widetilde X$},
\end{equation}
where $C>0$ is a constant depending only on $(X,\omega)$. 
Estimates \eqref{e:Bexp00} and \eqref{e:Bexp01} show  
that the asymptotics of the Bergman kernels 
$\widetilde P_p(\widetilde x)$ and $P_p(x)$, $x=q(\widetilde x)$, 
are the same. For $(L_p,h_p)=(L^p,h^p)$ with a smooth metric 
$h$ satisfying $c_1(L,h)\geq a\omega$ this follows from 
\cite[Theorem\,6.1.4]{MM07}.


\end{document}